\documentclass[12pt]{amsart}
\usepackage{epsfig}
\usepackage{amsfonts, amssymb, latexsym, amsmath, amscd}
\usepackage[mathscr]{eucal}
\usepackage{enumerate}
\usepackage[roman]{sublabel}
\usepackage{ottmacros}

\theoremstyle{plain}
\newtheorem{theorem}{Theorem}[section]
\newtheorem{corollary}{Corollary}[section]
\newtheorem{lemma}{Lemma}[section]
\newtheorem{proposition}{Proposition}[section]

\theoremstyle{definition}
\newtheorem{definition}{Definition}[section]
\newtheorem{notation}{Notation}[section]
\newtheorem{remark}{Remark}[section]

\newcounter{grealm}
\newcounter{hrealm}
\newcounter{sarealm}
\newcounter{forms}
\newcounter{derivaeq}
\newcounter{psect}
\newcounter{cscan}

\renewcommand{\thehrealm}{H\arabic{hrealm}}

\numberwithin{equation}{section}

\begin{document}

\title{Dissipative homoclinic loops and rank one chaos}

\author{Qiudong Wang}
\address{University of Arizona}
\email[Qiudong Wang]{dwang@math.arizona.edu}
\urladdr[Qiudong Wang]{www.math.arizona.edu/$\sim$dwang}

\author{William Ott}
\address{Courant Institute of Mathematical Sciences}
\urladdr[William Ott]{www.cims.nyu.edu/$\sim$ott}


\date{\today}

\begin{abstract}
We prove that when subjected to periodic forcing of the form
$p_{\mu, \rh, \om} (t) = \mu (\rh h(x,y) + \sin (\om t))$, certain
second order systems of differential equations with dissipative
homoclinic loops admit strange attractors with SRB measures for a
set of forcing parameters $(\mu, \rh, \om)$ of positive measure.
Our proof applies the recent theory of rank one maps, developed by
Wang and Young~\cite{WqYls2001, WqYls2008} based on the analysis
of strongly dissipative H\'enon maps by Benedicks and Carleson
~\cite{BmCl1985, BmCl1991}.
\end{abstract}

\maketitle

\section{Introduction}\label{s0}
In this paper we establish connections between a recent dynamics theory, namely the theory
of rank one maps, and a classical dynamical scenario, namely periodic perturbations of
homoclinic solutions. We prove that when subjected to periodic forcing of the form
$p_{\mu, \rho, \omega}(t) = \mu (\rho h(x, y) + \sin \omega t)$, certain second order
equations with a dissipative homoclinic saddle admit strange attractors with SRB measures
for a positive measure set of forcing parameters $(\mu, \rho, \omega)$.

\smallskip

\noindent {\bfseries \sffamily A. The theory of rank one maps.} \
The theory of rank one maps, systematically developed by Wang and Young~\cite{WqYls2001,
  WqYls2008}, concerns the dynamics of maps with some instability in one direction of the
phase space and strong contraction in all other directions of the
phase space. This theory originates from the work of
Jackboson~\cite{Jm1981} on the quadratic family $f_a(x) = 1-ax^2$
and the {\it tour de force} analysis of strongly dissipative
H\'enon maps by Benedicks and Carelson~\cite{BmCl1991}.

The theory of 1D maps with critical points has progressed
dramatically over the last 30 years~\cite{Mi1981, Jm1981, CE1983,
BmCl1985, TpTcYls1994}. The breakthrough from $1$D maps to $2$D
maps is due to Benedicks and Carleson ~\cite{BmCl1985, BmCl1991}.
Based on ~\cite{BmCl1991}, SRB measures were constructed for the
first time in ~\cite{BmYls1993} for a (genuinely) nonuniformly
hyperbolic attractor. The results in~\cite{BmCl1991} were
generalized in ~\cite{MlVm1993} to small perturbations of
H\'{e}non maps. These papers form the core material referred to in
the second box below.

\vspace{.3in}

\hspace{.7cm} \framebox[3cm][c]{$\begin{array}{c}
\text{Theory of}  \\
\text{$1$D maps}\end{array}$} \ {\Large $\longrightarrow$} \
\framebox[3cm][c]{$\begin{array}{c}
\text{H\'enon maps} \\
\text{ \& perturbations\ }
\end{array}$}
\ {\Large $\longrightarrow$} \ \framebox[3cm][c]{$\begin{array}{c}
\text{\it Rank one} \\
\text{\it attractors}
\end{array}$}

\vspace{.3in}

All of the results in the second box depend on the formula of the H\'enon maps. In going
from the second box to the third box, the authors of~\cite{WqYls2001} and~\cite{WqYls2008}
have aimed at developing a comprehensive chaos theory for a nonuniformly hyperbolic
setting that is flexible enough to be applicable to concrete systems of differential
equations.

The theory of rank one maps has been applied to various systems of
ordinary differential equations~\cite{WqYls2002, WqYls2003,
Lk2006, GjWmYls2006, Ow2007}.  The most siginificant application
thus far has been the analysis of periodically-kicked limit cycles
and Hopf bifurcations~\cite{WqYls2002, WqYls2003}.  In these
cases, periodic kicks of the limit cycle and separated by long
periods of relaxation to the limit cycle.  If the contraction to
the limit cycle is weak and the shear is strong, then admissible
families of rank one maps are produced.  The analysis of
periodically-kicked Hopf limit cycles has been extended to the
setting of parabolic partial differential
equations~\cite{LkWqYls2008}. These studies illustrate that the
theory of rank one maps can be used to rigorously prove the
existence of strange attractors with SRB measures for physically
meaningful differential equations.

\smallskip

\noindent {\bfseries \sffamily B.  Periodically-perturbed
  homoclinic solutions.} \ Periodically-forced second order systems, such as the
periodically-perturbed nonlinear pendulum, Duffing's equation, and van der Pol's equation
have been studied extensively in the past ~\cite{Av1989, Dg1918, GjHp1983, Lm1981, Ln1949,
  vdPb1927}. When a given second order equation with a homoclinic saddle is periodically
perturbed, the stable and unstable manifolds of the perturbed saddle intersect
transversely within a certain range of forcing parameters, generating homoclinic tangles
and chaotic dynamics~\cite{Mv1963, Ph1957a, Ph1957b, Ph1957c, Ss1967}.  Homoclinic tangles
were first observed by H. Poincar\'e ~\cite{Ph1957a, Ph1957b, Ph1957c}.  There also exist
parameters for which the stable and unstable manifolds of the perturbed saddle are pulled
apart.  For these two cases, Fig. 1 schematically illustrates the time-$T$ maps for the
perturbed equations, where $T$ is the period of the perturbation. The first picture leads
to exceedingly messy dynamics and the second appears simple.

\begin{picture}(12, 4)
\put(2,0){ \psfig{figure=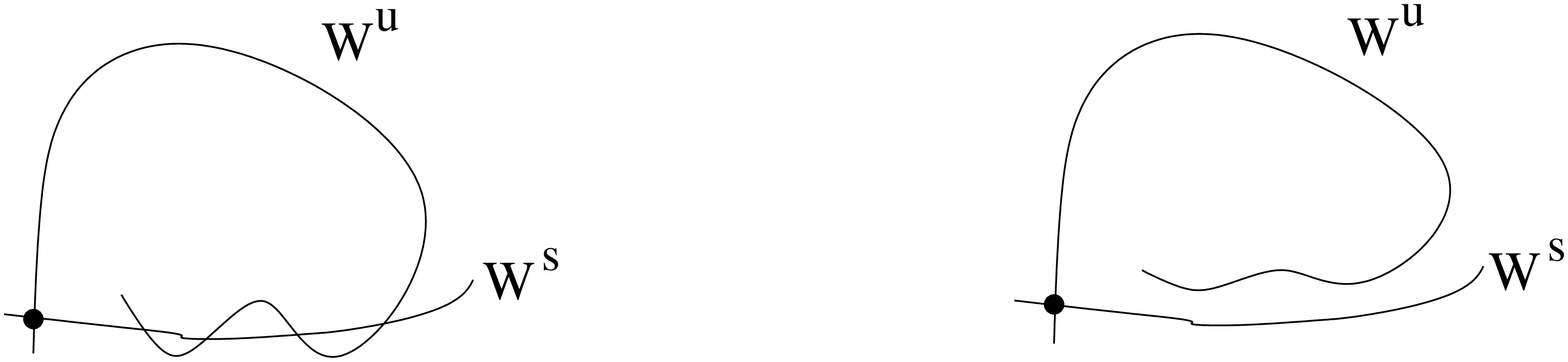,height = 3cm, width =
11cm} }
\end{picture}

\centerline{\ \ \ (a)  \hspace{2.5in} (b)}

\smallskip

\centerline{Fig. 1 \ (a) Homoclinic intersections and (b)
separated invariant manifolds.}

\medskip

In this paper we study periodically-perturbed second order equations but we follow a new
route. Instead of looking at the time-$T$ maps, we extend the phase space to three
dimensions and we {\it explicitly} compute return maps induced by the perturbed equations
in a neighborhood of the extended homoclinic solution.  To be more precise, we use
variables $(x, y)$ to represent the phase space of the unperturbed equation and we let
$(x, y) = (0, 0)$ be the saddle fixed point. Write the homoclinic solution for $(x, y) =
(0, 0)$ as $\ell$. We construct a small neighborhood of $\ell$ by taking the union of a
small neighborhood $U_{\varepsilon}$ of $(0, 0)$ and a small neighborhood $D$ around
$\ell$ outside of $U_{\frac{1}{4} \varepsilon}$. See Fig. 2. Let $\sigma^{\pm} \in
U_{\varepsilon} \cap D$ be the two line segments depicted in Fig. 2, both of which are
perpendicular to the homoclinic solution. We use the angular variable $\theta \in \Sbb^1$
to represent the time.

\begin{picture}(7, 5.5)
\put(4,0){ \psfig{figure=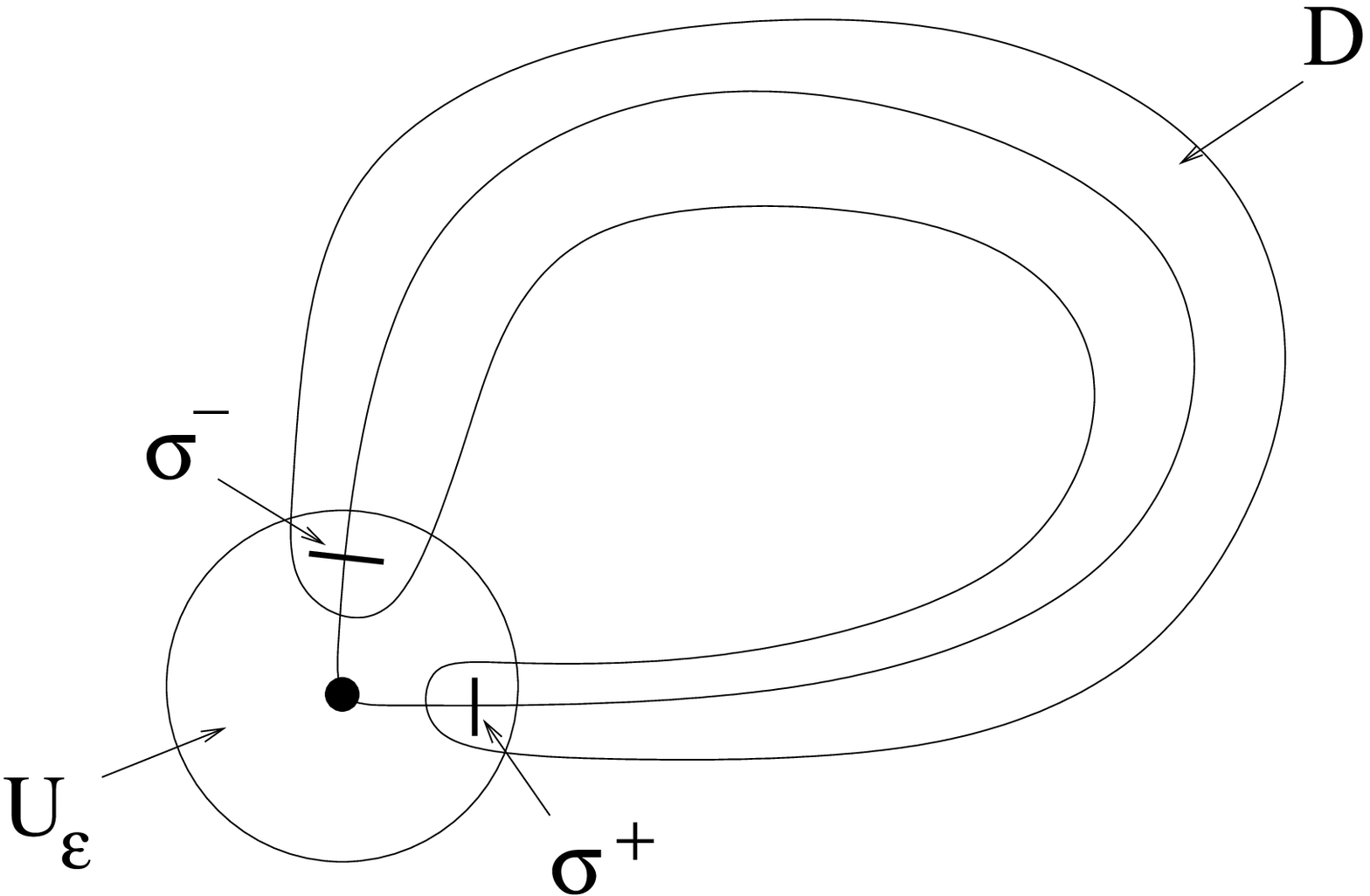,height = 5cm, width = 7cm}
}
\end{picture}

\medskip

\centerline{Fig. 2 \ $U_{\varepsilon}$, $D$ and $\sigma^{\pm}$.}

\medskip

In the extended phase space $(x, y, \theta)$ we define
\begin{equation*}
\mscr{U}_{\varepsilon} = U_{\varepsilon} \times \Sbb^1, \quad {\bf
D} = D \times \Sbb^1
\end{equation*}
and we let
\begin{equation*}
\Sigma^{\pm} = \sigma^{\pm} \times \Sbb^1.
\end{equation*}
Let $\mscr{N}: \Sigma^+ \to \Sigma^-$ be the map induced by the solutions in
$\mscr{U}_{\varepsilon}$ and let $\mscr{M}: \Sigma^- \to \Sigma^+$ be the map induced by
the solutions in ${\bf D}$. See Fig. 3.  We first compute $\mscr{M}$ and $\mscr{N}$
separately.  We then compose $\mscr{N}$ and $\mscr{M}$ to obtain an explicit formula for
the return map $\mscr{N} \circ \mscr{M}: \Sigma^- \to \Sigma^-$.  We show that for a large
open set of forcing parameters, these return maps naturally fall into the category of the
rank one maps studied in~\cite{WqYls2001} and~\cite{WqYls2008}.

\begin{picture}(8, 6)
\put(3.5,0){ \psfig{figure=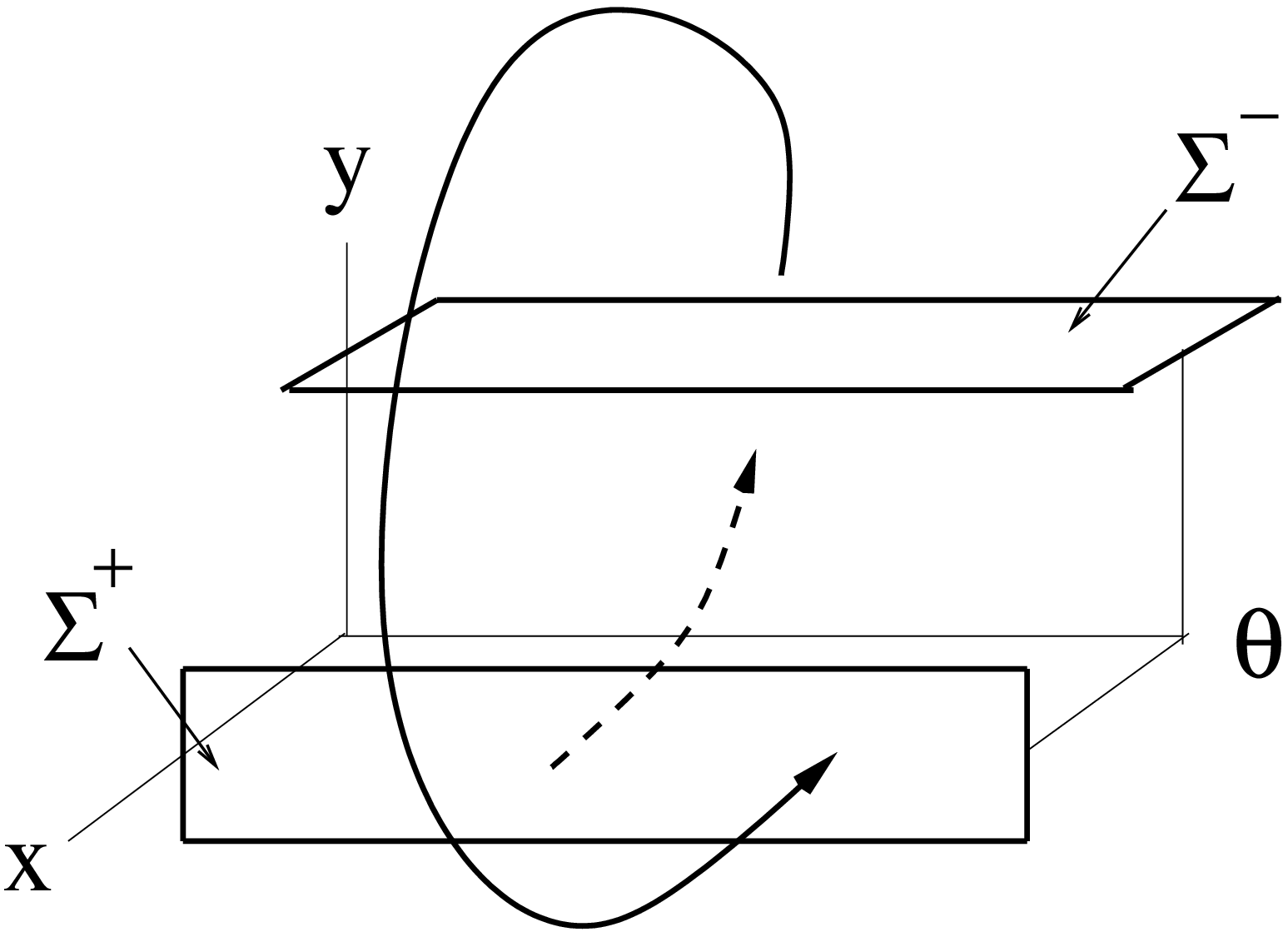,height = 5cm, width =
8cm} }
\end{picture}

\medskip

\centerline{Fig. 3 \ Construction of return maps.}

\medskip

\noindent {\bfseries \sffamily C. A brief summary of results.} \
Autonomous second order systems with a dissipative homoclinic
saddle are subjected to periodic forcing of the form $p_{\mu,
\rho, \omega}(t) = \mu (\rho h(x, y) + \sin \omega t)$, where
$\mu$, $\rho$, and $\omega$ are forcing parameters. We prove that
if the saddle is dissipative and nonresonant (see~\bpref{li:h1} in
Section~\ref{s1}) and if the unperturbed equation satisfies
certain nondegeneracy conditions (see~\bpref{li:h2} in
Section~\ref{s1}), then there exists an interval $[\rho_1,
\rho_2]$ such that for $\rh \in [\rh_{1}, \rh_{2}]$, the family of
return maps
\begin{equation*}
\{ (\mscr{N} \circ \mscr{M})_{\mu} : \mu \text{ is sufficiently small} \}
\end{equation*}
is a family of rank one maps to which the theory of~\cite{WqYls2001} and~\cite{WqYls2008}
directly applies.  In this parameter range, the stable and unstable manifolds of the
perturbed saddle do not intersect.  The dynamical properties of the periodically-perturbed
equations are determined by the magnitude of the forcing frequency $\omega$. When the
forcing frequency $\omega$ is small, there exists an attracting torus in the extended
phase space for all $\mu$ sufficiently small. In particular, there exists an attracting
torus consisting of quasiperiodic solutions for a set of $\mu$ with positive Lebesgue
density at $\mu = 0$. As $\omega$ increases, the attracting torus is disintegrates into
isolated periodic sinks and saddles. Increasing the magnitude of the forcing frequency
$\omega$ further, the phase space is stretched and folded, creating horseshoes and strange
attractors.  We prove in particular that these are strange attractors with SRB measures.
SRB measures represent visible statistical law in nonuniformly hyperbolic systems. The
chaos associated with them is both sustained in time and observable.  First constructed
for uniformly hyperbolic systems by Sinai~\cite{Sy1972}, Ruelle~\cite{Rd1976}, and
Bowen~\cite{Br1975}, SRB measures are the measures most compatible with volume when the
volume is not preserved. See~\cite{Yls2002} for a review of the theory and applications of
SRB measures.

In this paper, we focus exclusively on the scenario of rank one chaos.  See Theorems 1 and
2 in~\cite{WqYls2002} for results concerning the other scenarios described above.  We
remark that~\cite{WqYls2001} not only proves the existence of SRB measures, but also
establishes a comprehensive dynamical profile for the maps with SRB measures.  This
profile includes a detailed description of the geometric structure of the attractor and
statistical properties such as exponential decay of correlations.  We have opted to limit
the statements of our theorems to the existence of SRB measures, but all aspects of this
larger dynamical profile apply.

This paper is not only about the generic existence of rank one attractors in
periodically-forced second order equations.  Explicit, verifiable conditions are
formulated.  Based on the theorems of this paper, the first named author has proven the
existence of rank one chaos in a Duffing equation of the form
\begin{equation*}
\frac{d^2 q}{dt^2} + (a - b q^2) \frac{d q}{dt} - q + q^3 = \mu
\sin \omega t
\end{equation*}
and in a periodically-forced pendulum of the form
\begin{equation*}
\frac{d^2 \theta}{d t^2} - \delta \frac{d \theta}{dt} + \sin
\theta =  \alpha + \mu \sin \omega t.
\end{equation*}
These results will be presented in separate papers.

The analysis in this paper is not sensitive to the particular form we have chosen for the
forcing. We work with the forcing function $p_{\mu, \rho, \omega}$ because the resulting
analysis is relatively transparent. A theorem analogous to our main theorem holds for a
general class of forcing functions.

\smallskip

This paper is organized as follows. We state our results precisely in Section~\ref{s1}. In
Section~\ref{s:modelas} we discuss a model of Afraimovich and Shilnikov.
Sections~\ref{s2}--\ref{s6} are devoted to the proof of the main theorem.

\medskip

\noindent {\bfseries \sffamily D.  Acknowledgment.} \ Our method is
motivated by a paper of V.S.  Afraimovich and L.P. Shil'nikov published almost thirty
years ago~\cite{AvSl1977}. Afraimovich and Shil'nikov observed that for
periodically-forced systems with dissipative homoclinic loops, the dissipation around the
fixed point could potentially put the flow-induced return maps into the category (in our
terminology) of rank one maps. In this paper we basically start from where they stopped,
turning an insightful observation into a theorem one can use to analyze concrete
equations.  We are deeply indebted to Afraimovich for bringing his previous work with
Shilnikov~\cite{AvSl1977} to our attention. See also~\cite{AvHsb2003}. We also thank
Kening Lu and Lai-Sang Young for motivating conversations related to this work, and
particularly Lai-Sang Young for connecting us to Afraimovich and his work with Shil'nikov.

\section{Statement of Results}\label{s1}

Let $(x, y) \in {\mathbb R}^2$ be the phase variables and $t$ be
the time. We start with an autonomous system
\begin{equation}\label{f1-s1.1}
\left\{
\begin{aligned}
\frac{dx}{dt} &= - \alpha x + f(x, y)\\
\frac{dy}{dt} &= \beta y + g(x, y)
\end{aligned}
\right.
\end{equation}
where $f$ and $g$ are real analytic at $(x, y) = (0, 0)$
and $f(0, 0) = g(0, 0) = \partial_x f(0, 0) = \partial_y f(0, 0) =
\partial_x g(0, 0) = \partial_y g(0, 0) = 0$. We assume that $\alpha$ and
$\beta$ satisfy a certain Diophantine nonresonance condition and that
$(x, y) = (0, 0)$ is a dissipative saddle point. Namely, we assume the following.

\vspace{0.3cm}
\refstepcounter{hrealm}
\label{li:h1}
\noindent
{\bfseries (\thehrealm) Nonresonant dissipative saddle.}
\begin{enumerate}[{\bfseries (a)}]
\item
\label{li:h1a}
There exist $d_1, d_2> 0$ such that for all
$m$, $n \in {\mathbb Z}^+$, we have
\begin{equation*}
|m \alpha - n \beta|
> d_1 (|m| + |n|)^{-d_2}.
\end{equation*}
\item
\label{li:h1b}
$0 < \beta < \alpha$.
\end{enumerate}
\vspace{0.3cm}
We also assume that the positive $x$-side of the local stable
manifold of $(0,0)$ and the positive $y$-side of the local unstable manifold
of $(0, 0)$ are included as part of a homoclinic solution which
we denote as $x = a(t)$, $y = b(t)$. Let
\begin{equation*}
\ell = \{ \ell(t)=(a(t), b(t)) \in {\mathbb R}^2 : t \in
{\mathbb R} \}.
\end{equation*}
We further assume that $f(x, y)$ and $g(x, y)$ are $C^4$ in a
sufficiently small neighborhood of $\ell$.

To the right side of equation ~(\ref{f1-s1.1}) we add a
time-periodic term to form a non-autonomous system
\begin{equation}\label{f2-s1.1}
\left\{
\begin{aligned}
\frac{dx}{dt} &= - \alpha x + f(x, y) - \mu (\rho h(x, y) + \sin
\omega t)\\
\frac{dy}{dt} &= \beta y + g(x, y) + \mu (\rho h(x, y) + \sin
\omega t)
\end{aligned}
\right.
\end{equation}
where $\mu$, $\rho$, and $\omega$ are parameters. We assume that $h(x, y)$ is analytic at
$(x, y) = (0, 0)$ and $C^4$ in a small neighborhood of the homoclinic loop $\ell$. The
parameter $\mu$ satisfies $0 \leqs \mu \ll 1$ and controls the magnitude of the forcing
term. The prefactor $\rho$ and the forcing frequency $\omega$ are much larger parameters,
the ranges of which we will make explicit momentarily. Observe that the same forcing
function is added to the equation for $y$ but subtracted from the equation for $x$.  We do
this to facilitate the application of our theorem to a certain concrete second order
system.  The analysis in this work is by no means limited to these particular forcing
functions.

To study ~(\ref{f2-s1.1}), we introduce an angular variable $\theta
\in \Sbb^{1}$ and write it as
\begin{equation}\label{f3-s1.1}
\left\{
\begin{aligned}
\frac{dx}{dt} &= - \alpha x + f(x, y) - \mu (\rho h(x, y) + \sin \theta) \\
\frac{dy}{dt} &= \beta y + g(x, y) + \mu (\rho h(x, y) + \sin \theta) \\
\frac{d \theta}{dt} &= \omega.
\end{aligned}
\right.
\end{equation}
We denote
\begin{equation*}
(u(t), v(t)) = \left\| \frac{d}{dt} \ell(t) \right\|^{-1} \frac{d}{dt}
\ell(t)
\end{equation*}
where $\ell(t) = (a(t), b(t))$ is the homoclinic loop of equation
~(\ref{f1-s1.1}). The vector $(u(t), v(t))$ is a unit vector
tangent to $\ell$ at $\ell(t)$. Define
\begin{equation}\label{f1a-s1.1}
\begin{aligned}
E(t) &= v^2(t) (- \alpha + \partial_x f(a(t), b(t))) +
u^2(t)(\beta +
\partial_y g(a(t), b(t))) \\
&\qquad {}- u(t) v(t) (\partial_y f(a(t), b(t)) + \partial_x g(a(t),
b(t))).
\end{aligned}
\end{equation}
The quantity $E(t)$ measures the rate of expansion of the solutions of equation
~(\ref{f1-s1.1}) in the direction normal to $\ell$ at $\ell(t)$
(see Section~\ref{s3.1}).  In matrix form, we have
\begin{equation*}
E(t) =
\begin{pmatrix}
v(t) & -u(t)
\end{pmatrix}
\begin{pmatrix}
-\al + \pdop{x} f(\ell (t)) & \pdop{y} f(\ell (t)) \\
\pdop{x} g(\ell (t)) & \be + \pdop{y} g(\ell (t))
\end{pmatrix}
\begin{pmatrix}
v(t) \\
-u(t)
\end{pmatrix}
\end{equation*}
Define
\begin{equation}\label{f2a-s1.1}
\begin{aligned}
A &= \int_{-\infty}^{\infty}(u(s) +  v(s)) h(a(s), b(s)) e^{-
\int_{0}^{s}
E(\tau) \, d\tau} \, ds  \\
C &= \int_{-\infty}^{\infty}(u(s) + v(s)) \cos (\omega s) e^{- \int_{0}^{s} E(\tau) \, d\tau} \, ds \\
S &= \int_{-\infty}^{\infty}(u(s) +  v(s)) \sin (\omega s) e^{-
\int_{0}^{s} E(\tau) \, d\tau} \, ds.
\end{aligned}
\end{equation}

The integrals $A$, $C$, and $S$ are all absolutely convergent (see
Lemma~\ref{lem3-s2.2b}).  They describe the relative positions of the stable and unstable
manifolds of the perturbed saddle. See Fig. 4.  The quantity $\rho A \mu$ measures the
average distance between the stable and unstable manifolds and $\mu (C^2 +
S^2)^{\frac{1}{2}}$ measures the magnitude of the oscillation of the unstable manifold
relative to the stable manifold.

\begin{picture}(7.5, 5.2)
\put(4,0){ \psfig{figure= 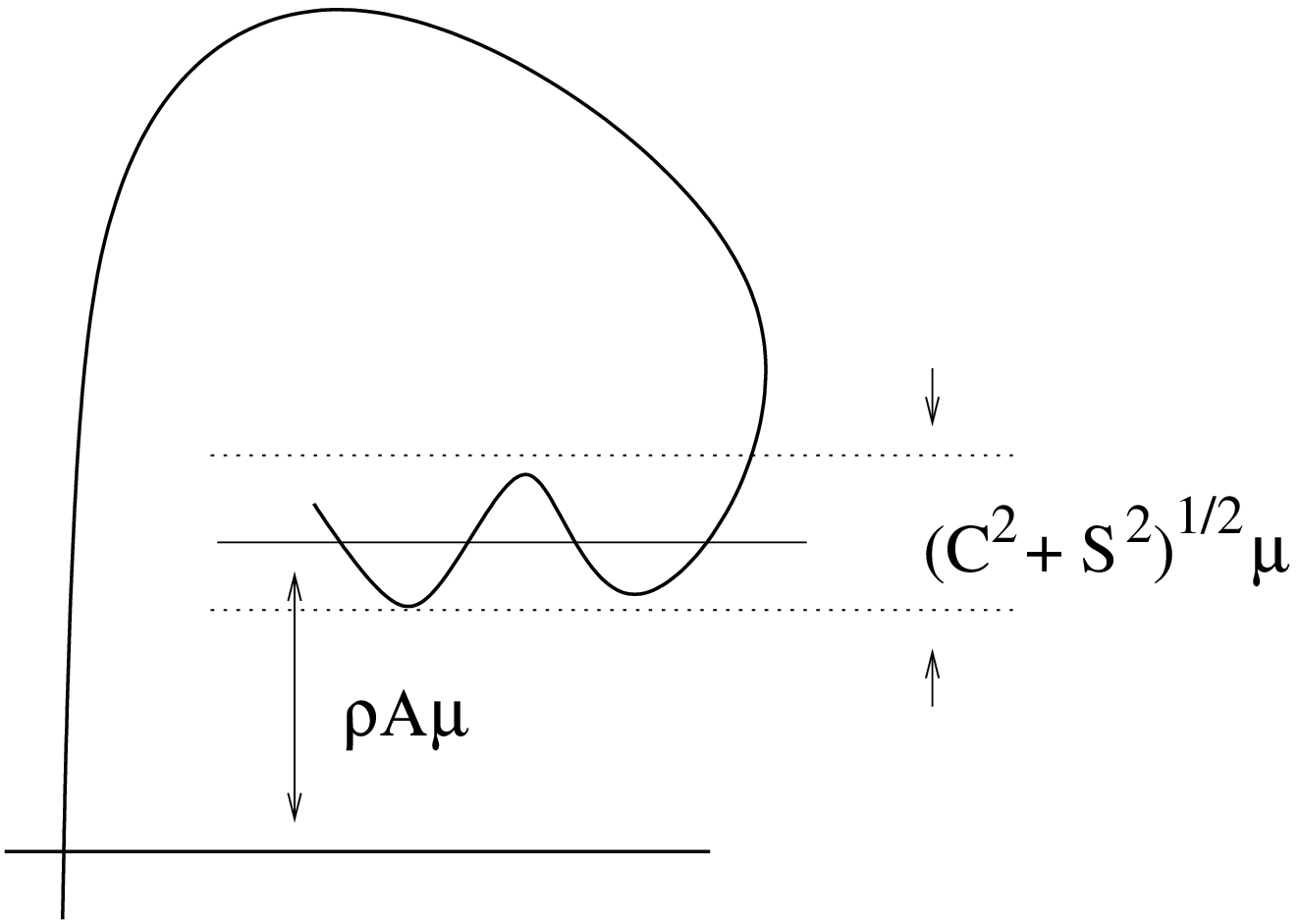, height = 4.5cm, width =
7cm} }
\end{picture}

\vskip .2in

\centerline{Fig. 4. \ The geometric meaning of the integrals $A$, $C$, and $S$.}

\vskip .2in

We assume that $A$, $C$, and $S$ satisfy the following
nondegeneracy conditions.

\vspace{0.3cm}
\refstepcounter{hrealm}
\label{li:h2}
\noindent
{\bfseries (\thehrealm) Nondegeneracy conditions on} $\mathbf{A}${\bfseries ,} $\mathbf{C}${\bfseries , and} $\mathbf{S}${\bfseries .}
\begin{enumerate}[{\bfseries (a)}]
\item \label{li:h2a} $A \neq 0$. \item \label{li:h2b} $C^2 + S^2
\neq 0$.
\end{enumerate}
\vspace{0.3cm}

Given equation ~(\ref{f2-s1.1}) satisfying~\bpref{li:h1} and~\bpref{li:h2}, we
let
\begin{equation*}
\rho_1 = - \frac{202}{99} \frac{\sqrt{C^2 + S^2}}{A}, \quad \rho_2
= - \frac{396}{101} \frac{\sqrt{C^2 + S^2}}{A}.
\end{equation*}
We also let
\begin{equation*}
I = \{ z \in {\mathbb R}, \ |z| < K \mu \}
\end{equation*}
for some $K > 1$ sufficiently large independent of $\mu$ and
\begin{equation*}
\Sigma  = \{\ell(0) + (v(0), - u(0)) z \in {\mathbb R}^2 : z \in I
\} \times \Sbb^{1}.
\end{equation*}
The following is the main theorem of this paper.
\begin{theorem}\label{th-s1.1} Assume that~\pref{f2-s1.1}
satisfies~\bpref{li:h1} and~\bpref{li:h2}\bpref{li:h2a}.  There
exists $\om_{0} > 0$ such that if $\om \in \RR$
satisfies~\bpref{li:h2}\bpref{li:h2b} and $|\om| > \om_{0}$, then
for every $\rh \in [\rh_{1}, \rh_{2}]$ we have the following.
\begin{enumerate}
\item For $\mu$ sufficiently small, equation~\pref{f3-s1.1} induces a well-defined return
  map $\mscr{F}_{\mu} : \Si \to \Si$.
\item There exists a set $\De_{\om, \rh}$ of values of $\mu$ with positive lower Lebesgue
  density at $\mu = 0$ such that for every $\mu \in \De_{\om, \rh}$, $\mscr{F}_{\mu}$
  admits a strange attractor that supports an ergodic SRB measure $\nu$.  Furthermore,
  Lebesgue almost every point on $\Si$ is generic with respect to $\nu$.
\end{enumerate}
\end{theorem}

\begin{picture}(5.5, 5.5)
\put(5,0){ \psfig{figure=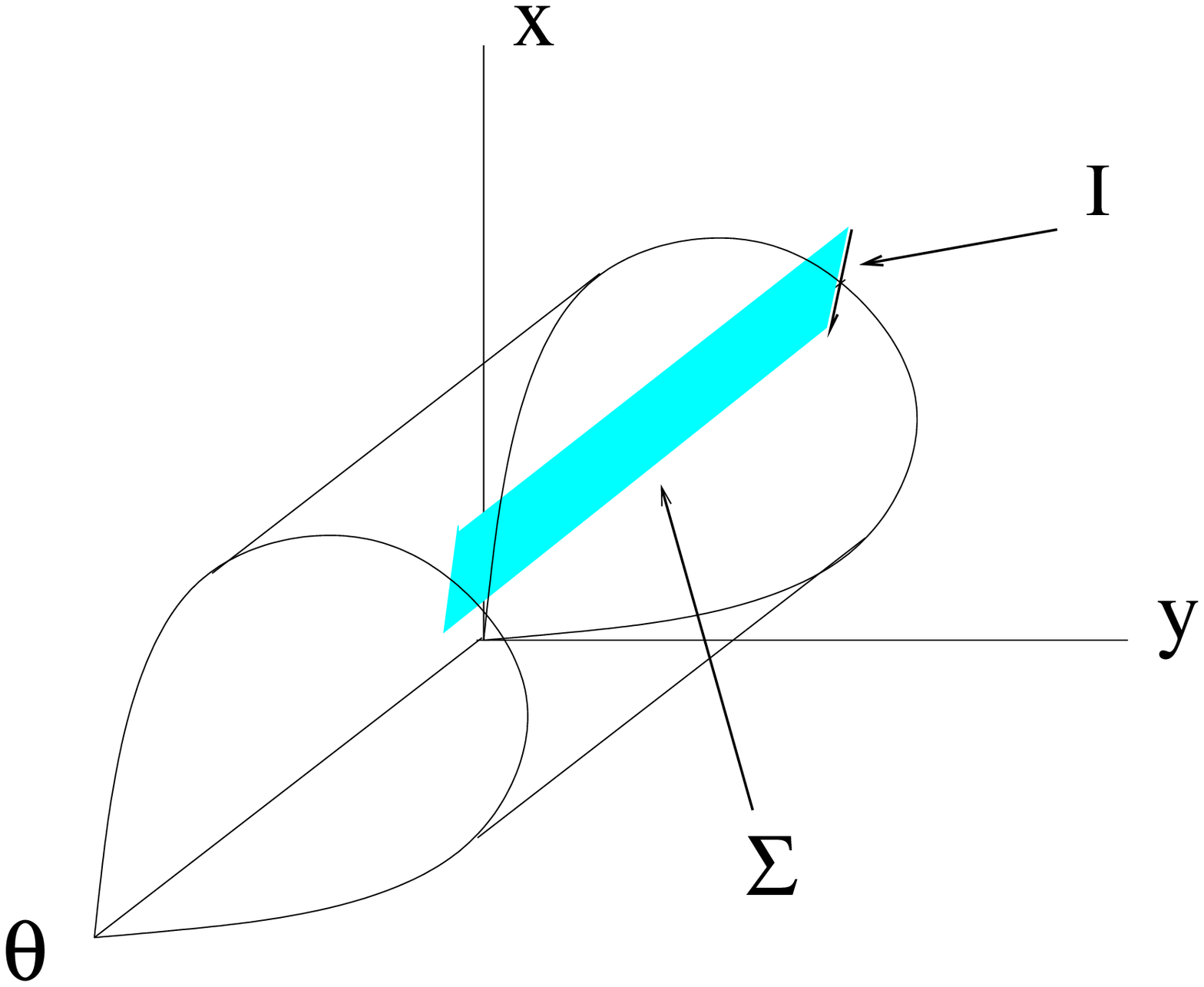,height = 5cm, width = 5cm}
}
\end{picture}

\vskip .2in

\centerline{Fig. 5. \ The Poincar\'e section $\Sigma$.}

\vskip .2in 

We recall that an $\mscr{F}$-invariant Borel
probability measure $\nu$ on $\Si$ is an {\bfseries \itshape SRB
measure} if $\mscr{F}$ has a positive Lyapunov exponent
$\nu$-almost everywhere and if the conditional measures of $\nu$
on unstable manifolds are absolutely continuous with respect to
the Riemannian measures on these unstable leaves. SRB measures
represent {\itshape visible statistical law} in chaotic systems.

\begin{remark}
As an important condition to be
verified,~\bpref{li:h2} does not cast doubt on the abundance of the type
of strange attractor proved to exist in this paper. By properly
adjusting the sign of $h(x, y)$ according to the sign of $u(s) +
v(s)$ on $\ell$, we can easily achieve $A \neq 0$.  Hypothesis~\bpref{li:h2}\bpref{li:h2b}
requires that the Fourier spectrum of the
function
\begin{equation*}
R(s) = (u(s) + v(s)) e^{-\int_0^s E(\tau) \, d \tau}
\end{equation*}
is not identically zero on the frequency range higher than
$\omega_0$. Since $R(s)$ decays exponentially as a function of
$s$, the Fourier transform $\hat R(\xi)$ is analytic in a strip
containing the real $\xi$-axis by the Paley-Wiener theorem.  It follows that $\hat R(\xi) = 0$ for
at most a discrete set of values of $\xi$ unless $R(s)$ is
identically zero.
\end{remark}

\section{A model of Afraimovich and Shilnikov}
\label{s:modelas}

In this section we study a model introduced by Afraimovich and
Shilnikov in~\cite{AvSl1977}. See also~\cite{AvHsb2003}. This
simple model allows us to illustrate the steps of the proof of the main theorem
without needing to deal with technical
complexity. The return maps of Afraimovich and Shilnikov are
derived in Section \ref{s3.1a}. In Section \ref{s3.2a} we prove
that these return maps are rank one maps in the sense of \cite{WqYls2001} and
\cite{WqYls2008}.

\subsection{Derivation of return maps}\label{s3.1a}
We begin by describing an unperturbed system of differential
equations. Let $f : \RR^{2} \to \RR$ and $g : \RR^{2} \to \RR$ be
$C^{\infty}$ functions and let $\al$, $\be \in \RR$ satisfy $0 <
\be < \al$. Define
\begin{equation}\label{e:asun}
\left\{
\begin{aligned}
\tdf{x}{t} &= -\al x + f(x,y)\\
\tdf{y}{t} &= \be y + g(x,y).
\end{aligned}
\right.
\end{equation}
We assume that the functions $f$ and $g$ satisfy $f(x,y) = g(x,y)
= 0$ for all $(x,y) \in B(\mathbf{0}, 2 \ve)$ where $0 < \ve < 1$.
This means that equation~\pref{e:asun} is linear in a neighborhood
of $\mathbf{0}$. We also assume that equation~\pref{e:asun} admits
a homoclinic solution $\ell = \{\ell(t) :
t \in \RR \}$ containing the segments $\{ (0,y) : 0 < y < 2 \ve \}$
and $\{ (x,0) : 0 < x < 2 \ve \}$.

Let $\Sbb^{1} = [0, 2 \pi)$ denote the unit circle and let $p, q :
\RR^{2} \times \Sbb^{1} \to \RR$ be $C^{\infty}$ functions such
that $p=q=0$ on $B (\mathbf{0}, 2 \ve) \times \Sbb^{1}$. We now
introduce the perturbed system
\begin{equation}\label{e:asfull}
\left\{
\begin{aligned}
\tdf{x}{t} &= \al x + f(x,y) + \mu p(x,y,\thet)\\
\tdf{y}{t} &= \be y + g(x,y) + \mu q(x,y,\thet)\\
\tdf{\thet}{t} &= \om.
\end{aligned}
\right.
\end{equation}
Here $\om \in \RR$ is the frequency of the forcing functions and
$\mu> 0$ represents the strength of the perturbation. We
assume that $\mu$ and $\ve$ satisfy $0 \leqs \mu \ll \ve < 1$.

The orbit $\ga = \{ (0,0,\thet) : \thet \in \Sbb^{1} \}$ is a hyperbolic
periodic orbit of equation \pref{e:asfull} for all $\mu$.  For
$\mu = 0$, $\Ga = \ell \times \Sbb^{1}$ is the stable manifold and
the unstable manifold of $\ga$. We define the Poincar\'{e}
sections
\begin{align*}
\Si^{-} &= \{ (x,y,\thet) : 0 \leqs x \leqs C_{1} \mu,  \; y = \ve, \;
\thet \in \Sbb^{1} \}\\
\Si^{+} &= \{ (x,y,\thet) : x = \ve, \; C_{2}^{-1} \mu \leqs y
\leqs C_{2} \mu, \; \thet \in \Sbb^{1} \}
\end{align*}
where $\mu \in [0, \mu_{0}]$, $C_{1} > 0$ is such that $C_{1}
\mu_{0} \ll \ve$, and $C_{2}$ is suitably chosen.  We study a
situation in which one can define flow-induced maps $\mscr{M} :
\Si^{-} \to \Si^{+}$ and $\mscr{N}: \Si^{+} \to \Si^{-}$ (see
Section \ref{s0}B and Fig. 3). The composition $\mscr{N} \circ
\mscr{M}$ produces a one-parameter family $\{ \mscr{F}_{\mu} =
\mscr{N} \circ \mscr{M} : \mu \in [0, \mu_{0}] \}$ of maps from
$\Si^{-}$ to $\Si^{-}$.

\smallskip

\noindent {\bf The map {\mathversion{bold} $\mscr{N} : \Si^{+} \to
\Si^{-}$}.} \ The flow from $\Si^{+}$ to $\Si^{-}$ is defined by
the differential equations
\begin{align}
\sublabon{equation}
\tdf{x}{t} &= -\al x
\label{e:Na}\\
\tdf{y}{t} &= \be y
\label{e:Nb}\\
\tdf{\thet}{t} &= \om.
\label{e:Nc}
\end{align}
\sublaboff{equation} Let $(\ve, \hat{y}, \hat{\thet}) \in
\Si^{+}$.  Let $T(\hat{y})$ denote the time at which the orbit
emanating from $(\ve, \hat{y}, \hat{\thet})$ intersects $\Si^{-}$.
Write $\mscr{N} (\ve, \hat{y}, \hat{\thet}) = (x_{1}, \ve,
\thet_{1})$. Integrating~\pref{e:Nb}, we have $\ve = e^{\be
T(\hat{y})} \hat{y}$, so $T(\hat{y}) = \frac{1}{\be} \log (\ve
\hat{y}^{-1})$.  Integrating~\pref{e:Na} yields $x_{1} = e^{-\al
T(\hat{y})} \ve = \ve^{1 - \tfrac{\al}{\be}}
\hat{y}^{\tfrac{\al}{\be}}$.  The local map $\mscr{N}$ is
therefore given by
\begin{equation}\label{e:Nfin}
\left\{
\begin{aligned}
x_{1} &= \ve^{1 - \tfrac{\al}{\be}} \hat{y}^{\tfrac{\al}{\be}}\\
\thet_{1} &= \hat{\thet} + \frac{\om}{\be} \log (\ve
\hat{y}^{-1}).
\end{aligned}
\right.
\end{equation}

\smallskip

\noindent {\bf The map {\mathversion{bold} $\mscr{M} : \Si^{-} \to
\Si^{+}$}.} \ Let $(x_{0}, \ve, \thet_{0}) \in \Si^{-}$.  Write
$\mscr{M} (x_{0}, \ve, \thet_{0}) = (\ve, \hat{y}, \hat{\thet})$.
We assume that for $\mu \in [0, \mu_{0}]$,
\begin{align*}
\hat{y} &= \la x_{0} + \mu \vp (x_{0}, \thet_{0})\\
\hat{\thet} &= \thet_{0} + \xi_{1} + \mu \psi (x_{0}, \thet_{0}).
\end{align*}
Here $0 < \la < 1$ and $\xi_{1} > 0$ are fixed.  The functions $\vp$
and $\psi$ are $C^{\infty}$ functions on $\Si^{-}$.  We assume that
$\vp (x_{0}, \thet_{0}) > 0$ for all $(x_{0}, \thet_{0}) \in \Si^{-}$.
This ensures that the second scenario of Fig. 1, namely the scenario
in which the stable and unstable manifolds are pulled apart by the
periodic forcing. More precisely, we assume that $p$ and $q$ are such
that
\begin{align*}
\psi (x_{0}, \thet_{0}) &= \xi_{2}\\
\vp (x_{0}, \thet_{0}) &= B (1 + A \sin \thet_{0}).
\end{align*}
Here $\xi_{2} \in \RR$, $B > 0$, and $0 < A < 1$. The global map $\mscr{M}$ is
therefore given by
\begin{equation}\label{e:Mfin}
\left\{
\begin{aligned}
\hat{y} &= \la x_{0} + \mu B (1 + A \sin \thet_{0})\\
\hat{\thet} &= \thet_{0} + \xi_{1} + \mu \xi_{2}.
\end{aligned}
\right.
\end{equation}
Let us not worry about the viability of these assumptions, enduring
for the moment the possibility that, at worst, no differential
equation satisfies all of our assumptions.

\smallskip

\noindent{\bf The map {\mathversion{bold} $\mscr{F}_{\mu} =
\mscr{N} \circ \mscr{M} : \Si^{-} \to \Si^{-}$}.} \ Let $(x_{0},
\ve, \thet_{0}) \in \Si^{-}$.  Computing $\mscr{F}_{\mu} (x_{0},
\ve, \thet_{0}) = (x_{1}, \ve, \thet_{1})$ using~\pref{e:Nfin}
and~\pref{e:Mfin}, we have
\begin{align*}
x_{1} &= \ve^{1 - \tfrac{\al}{\be}} \big[ \la x_{0} +
\mu B (1 + A \sin (\thet_{0}) \big]^{\tfrac{\al}{\be}}\\
\thet_{1} &= \thet_{0} + \xi_{1} + \mu \xi_{2}
+ \frac{\om}{\be} \log \left( \frac{\ve}{\la x_{0} +
\mu B (1 + A \sin (\thet_{0}))} \right).
\end{align*}
Using the spatial rescaling $x \mapsto \mu X$, we obtain
\begin{align}
\sublabon{equation}
X_{1} &= \ve^{1 - \tfrac{\al}{\be}} \mu^{\tfrac{\al}{\be} - 1} \big[ \la X_{0}
+ B (1 + A \sin (\thet_{0}) \big]^{\tfrac{\al}{\be}}
\label{e:FmuX}\\
\thet_{1} &= \thet_{0} + \xi_{1} + \mu \xi_{2}
+ \frac{\om}{\be} \log \left( \frac{\ve \mu^{-1}}{\la X_{0}
+ B (1 + A \sin (\thet_{0}))} \right).
\label{e:Fmuthet}
\end{align}
\sublaboff{equation}

Using this formula for $\mscr{F}_{\mu}$, Afraimovich and Shil'nikov
~\cite{AvSl1977, AvHsb2003} conclude that $\mscr{F}_{\mu}$ has a
horseshoe for large $\omega$.

\subsection{Theory of rank one attractors} \label{s3.2a}
In this subsection we first introduce admissible rank one maps following~\cite{WqYls2008}
and we then prove that $\{ \mscr{F}_{\mu} \}$ is an admissible family of rank one maps
using the techniques of~\cite{WqYls2002}.

\smallskip

\noindent {\bfseries \sffamily A.  Misuirewicz maps and admissible
  1D families.} \ The definition of an admissible family of 1D maps is rather long and
technical.  It could therefore present a nontrivial hurdle for the reader. We feel
obligated to present this definition for completeness.  Readers wishing to skip the
material on admissible 1D families can safely jump to Proposition~\ref{prop2-s4.2}.
Proposition~\ref{prop2-s4.2} contains the only result from the 1D aspect of rank one
theory that we need for the results of this paper.

We start with Misuirewicz maps. For $f \in C^{2} (\Sbb^{1}, \Sbb^1)$, let $C= C(f) =
\{f'=0\}$ denote the critical set of $f$ and let $C_\delta$ denote the
$\delta$-neighborhood of $C$ in $\Sbb^1$. For $x \in \Sbb^1$, let $d(x,C) = \min_{\hat x
  \in C} |x-\hat x|$.

\begin{definition}
\label{def-1.1} We say that $f \in C^2(\Sbb^1, \Sbb^1)$ is a
Misuirewicz map and we write $f \in \mscr{E}$ if the following
hold for some $\delta_0>0$.

\begin{enumerate}[(1)]
\item
\label{d:mis1}
{\mathversion{bold} $\text{Outside of } C_{\de_{0}}.$}
There exist
$\lambda_0 > 0$, $M_0 \in {\mathbb Z}^+$, and $0< c_0 \leq 1$ such that
\begin{enumerate}[(a)]
\item
\label{d:mis1a}
for all $n \geq M_0$, if $x, f(x), \cdots, f^{n-1}(x) \not \in
C_{\delta_0}$, then $|(f^n)'(x)| \geqs e^{\lambda_0 n}$;
\item
\label{d:mis1b}
if $x, f(x), \cdots, f^{n-1}(x) \not \in C_{\delta_0}$ and
$f^n(x) \in C_{\delta_0}$ for any $n$, then $|(f^n)'(x)| \geq c_0
e^{\lambda_0 n}$.
\end{enumerate}
\item
\label{d:mis2}
{\mathversion{bold} $\text{Inside } C_{\de_{0}}.$}
\begin{enumerate}[(a)]
\item
\label{d:mis2a}
We have $f''(x) \neq 0$ for all $x \in C_{\delta_0}$.
\item
\label{d:mis2b}
For all $\hat x \in C$ and $n > 0$, $d(f^n(\hat x), C) \geq
\delta_0$.
\item
\label{d:mis2c}
For all $x \in C_{\delta_0}\setminus C$, there exists
$p_0(x)>0$ such that $f^j(x) \not \in C_{\delta_0}$ for all
$j<p_0(x)$ and $|(f^{p_0(x)})'(x)| \geq$ $c_0^{-1} e^{\frac{1}{3}
\lambda_0 p_0(x)}$.
\end{enumerate}
\end{enumerate}
\end{definition}

We remark that Misurewicz maps are among the simplest maps with nonuniform expansion. The
phase space is divided into two regions, $C_{\delta_0}$ and $\Sbb^1 \setminus
C_{\delta_0}$. Condition~\pref{d:mis1} in Definition~\ref{def-1.1} says that on $\Sbb^1
\setminus C_{\delta_0}$, $f$ is essentially uniformly expanding. Condition~\pref{d:mis2c}
says that for $x \in C_{\delta_0} \setminus C$, even though $|f'(x)|$ is small, the orbit
of $x$ does not return to $C_{\delta_0}$ again until its derivative has regained a
definite amount of exponential growth.  In particular, if $n$ is the first return time of
$x \in C_{\delta_0}$ to $C_{\delta_0}$, then $|(f^n)'(x)| \geq c_{0}^{-1} e^{\frac{1}{3}
  \lambda_0 n}$.

We now define {\bfseries \itshape admissible families} of 1D maps. Let $F : \Sbb^1 \times
[a_{1}, a_{2}] \to \Sbb^1$ be a $C^{2}$ map. The map $F$ defines a one-parameter family
$\{ f_{a} \in C^{2} (\Sbb^1, \Sbb^1) : a \in [a_{1}, a_{2}] \}$ via $f_{a} (x) =
F(x,a)$. We assume that there exists $a^{*} \in (a_{1}, a_{2})$ such that $f_{a^{*}} \in
\mscr{E}$. For each $c \in C(f_{a^{*}})$, there exists a continuation $c(a) \in C(f_{a})$
provided $a$ is sufficiently close to $a^{*}$.

Let $C(f_{a^{*}}) = \{ c^{(1)} (a^{*}), \ldots, c^{(q)} (a^{*}) \}$, where $c^{(i)}
(a^{*}) < c^{(i+1)} (a^{*})$ for $1 \leqslant i \leqslant q-1$.  For $c(a^{*}) \in
C(f_{a^{*}})$, we define $\beta(a^{*}) = f_{a^{*}} (c(a^{*}))$.  For all parameters $a$
sufficiently close to $a^{*}$, there exists a unique continuation $\beta(a)$ of
$\beta(a^{*})$ such that the orbits
\begin{equation*}
\{ f_{a^{*}}^{n} (\beta(a^{*})) : n \geqslant 0 \} \text{ and } \{
f_{a}^{n} (\beta(a)) : n \geqslant 0 \}
\end{equation*}
have the same itineraries with respect to the partitions of $\Sbb^1$ induced by
$C(f_{a^{*}})$ and $C(f_{a})$.  This means that for all $n \geqslant 0$, $f_{a^{*}}^{n}
(\beta(a^{*})) \in (c^{(j)} (a^{*}), c^{(j+1)} (a^{*}))$ if and only if $f_{a}^{n}
(\beta(a)) \in (c^{(j)} (a), c^{(j+1)} (a))$ (here $\psup{c}{q+1} =
\psup{c}{1}$). Moreover, the map $a \mapsto \beta(a)$ is differentiable (see Proposition
4.1 in~\cite{WqYls2006}).

\begin{definition}\label{D:admissible}
Let $F : \Sbb^1 \times [a_{1}, a_{2}] \to \Sbb^1$ be a $C^{2}$
map. The associated one-parameter family $\{ f_{a} : a \in [a_{1},
a_{2}] \}$ is {\bfseries \itshape admissible} if
\begin{enumerate}[(1)]
\item
there exists $a^{*} \in (a_{1}, a_{2})$ such that $f_{a^{*}}
\in \mscr{E}$;
\item
for all $c \in C(f_{a^{*}})$, we have
\begin{equation}
\xi (c) = \left. \frac{d}{da} (f_{a} (c (a)) - \beta(a))
\right|_{a=a^{*}} \neq 0.
\end{equation}
\end{enumerate}
\end{definition}

The next proposition contains all that we need from the 1D aspect of rank one theory for
this paper.

\begin{proposition}[\cite{WqYls2002, WqYls2003, LkWqYls2008}]\label{prop2-s4.2}
Let $\Psi(\theta): \Sbb^1 \to {\mathbb R}$ be a $C^3$ function
with non-degenerate critical points and let $\Phi(\theta, a): \Sbb^1
\times [a_0, a_1] \to {\mathbb R}$ be such that
\begin{equation*}
\|\Phi(\theta, a)\|_{C^3 (\mbb{S}^{1} \times [a_{0}, a_{1}])} < \frac{1}{100}.
\end{equation*}
We define a one parameter family of circle maps $\{f_a: a \in [0,
2 \pi] \}$ by
\begin{equation*}
f_a(\theta) = \theta + \phi(\theta, a) + a + \mscr{K}
\Psi(\theta)
\end{equation*}
where $\mscr{K}$ is a constant.  There exists $K$, determined by $\Psi$ alone, such that
if $\mscr{K} > K$, then $\{ f_a \}$ is an admissible family of 1D maps.
\end{proposition}
The special case of this proposition in which $\Phi(\theta, a) = 0$ was first proved in
~\cite{WqYls2002}. That proof can easily be extended to prove Proposition
~\ref{prop2-s4.2}. See also Proposition 2.1 in ~\cite{WqYls2003} and Appendix C in
~\cite{LkWqYls2008}.

\smallskip

\noindent {\bfseries \sffamily B.  Admissible families of rank one
  maps.} \ We now move to the 2D part of the setting of~\cite{WqYls2001}
and~\cite{WqYls2008}.  Let $I$ be an interval.  Let $B_{0} \subset \mbb{R}$ be a set with
a limit point at $0$.  A 2-parameter $C^3$ family $\{ F_{a, b}(X, \theta) : a \in [a_{0},
a_{1}], \; b \in B_{0} \}$ of 2D diffoemorphisms defined on $\Si= I \times \Sbb^1$ is an
admissible rank one family if the following hold.

\smallskip

\refstepcounter{cscan} \label{c1} \bpref{c1} There exists a $C^2$
function $F_{a, 0} (X, \theta)$ of  $(a, X, \theta)$ such that, as
$b \to 0$,
\begin{equation*}
\|F_{a, b}(X, \theta) - (0, F_{a, 0} (X, \theta))\|_{C^3 ([a_{0},
a_{1}] \times \Si)} \to 0.
\end{equation*}

\refstepcounter{cscan} \label{c2} \bpref{c2} $\{ f_a(\theta) =
F_{a, 0} (0, \theta) : a \in [a_{0}, a_{1}] \}$ is an admissible
1D family.

\refstepcounter{cscan} \label{c3}
\bpref{c3} For all $a \in
[a_{0}, a_{1}]$, at the critical points of the 1D map
$f_a(\theta)$ we have
\begin{equation*}
\left. \frac{\partial}{\partial X}F_{a, 0} (X, \theta) \right|_{X = 0} \neq 0.
\end{equation*}

\smallskip

The following is the main result of~\cite{WqYls2001} and
~\cite{WqYls2008} for a given admissible rank one family $F_{a, b}$
of 2D maps.

\begin{proposition}[\cite{WqYls2001, WqYls2008}] \label{prop3-s3.2}
Let $F_{a, b}: \Si \to
\Si$ be an admissible rank one family.  There exists $\hat{b} > 0$ such that for all $|b|
< \hat{b}$, there exists a set $\Delta_b$ of values of $a$ with positive Lebesgue measure
such that for $a \in \Delta_b$, $F_{a, b}$ admits an ergodic SRB measure $\nu$. If we also
have $\lambda_0 > \ln 10$, where $\lambda_0$ is as in Definition \ref{def-1.1}, then $\nu$
is the only ergodic SRB measure\footnote{This is proved in~\cite{WqYls2002}.} that $F_{a,
  b}$ admits on $\Si$.
\end{proposition}

More is true if the global distortion bound~\bpref{c4} holds.

\smallskip

\refstepcounter{cscan} \label{c4}
\bpref{c4} There exists $C > 0$
such that for all $a \in [a_{0}, a_{1}]$, $b \in B_{0}$, $(X,
\thet) \in \Si$, and $(X', \thet') \in \Si$, we have
\begin{equation*}
\left|\frac{\det{D F_{a, b}(X, \theta)}}{\det{D F_{a, b}(X',
\theta')}}\right| < C.
\end{equation*}

\smallskip

\begin{proposition}[\cite{WqYls2001}]\label{prop4-s3.2}
  Let $F_{a,b}$ be an admissible rank one family satisfying~\bpref{c4} and suppose that
  $\la_{0} > \ln 10$, where $\la_{0}$ is as in Definition~\ref{def-1.1}.  Then for all
  $|b| < \hat{b}$ and $a \in \De_{b}$, Lebesgue almost every point in $\Si$ is generic
  with respect to the unique ergodic SRB measure on $\Si$.
\end{proposition}

\smallskip

\noindent
{\bfseries C.} {\mathversion{bold} $\{ \mscr{F}_{\mu} \}$} {\bfseries is an
  admissible family of rank one maps.} \ We show that $\{ \mscr{F}_{\mu} \}$ satisfies the
hypotheses~\bpref{c1}--\bpref{c4}.  Letting $\mu \to 0$ in~\pref{e:FmuX}, we see that
$X_{1} \to 0$ because $\al > \be$.  However, the term $\frac{\om}{\be} \log (\mu^{-1})
\pmod{2 \pi}$ fails to converge as $\mu \to 0$. The fact that $\thet_{1}$ is computed
modulo $2 \pi$ allows us to introduce the parameter $a$ and thereby obtain a two-parameter
family $\{ F_{a,b} \}$ with a well-defined 1D singular limit.

We regard $p = \log (\mu^{-1})$ as the fundamental parameter
associated with
$\{ \mscr{F}_{\mu} \}$.  Notice that we now have
$p \in [\log (\subsup{\mu}{0}{-1}), \infty)$.  Think of $\mu = e^{-p}$ as a
function of $p$.  Define $\ga : (0, \mu_{0}] \to \RR$ by
\begin{equation*}
\ga (\mu) = \frac{\om}{\be} \log (\mu^{-1}).
\end{equation*}
Let $N \in \NN$ satisfy $\frac{\om}{\be} \log (\subsup{\mu}{0}{-1}) < N$.
Let $(\mu_{n})$ be the decreasing sequence of values of $\mu$ such that
$\ga (\mu_{n}) = N + 2 \pi (n-1)$ for every $n \in \NN$.  We think of $\mu$ as a
measure of dissipation and we therefore set $b_{n} = \mu_{n}$.
For $a \in \Sbb^{1}$ and $n \in \NN$, define
\begin{align*}
\mu (n,a) &= \ga^{-1} (\ga (\mu_{n}) + a)\\
p(n,a) &= \log (\mu (n,a)^{-1}) = \log (\subsup{\mu}{n}{-1})
+ \frac{\be}{\om} a.
\end{align*}
The map $F_{a,b_{n}}$ is defined by $F_{a,b_{n}} =
\mscr{F}_{p(n,a)}$.

The family $\{ F_{a,b_{n}} \}$ has a well-defined singular limit.
As $n \to \infty$, $F_{a,b_{n}}$ converges in the $C^{3}$ topology
to the map $F_{a,0}$ defined by
\begin{align*}
\subasup{F}{a,0}{1} (X_{0}, \thet_{0}) &= 0\\
\subasup{F}{a,0}{2} (X_{0}, \thet_{0}) &= \thet_{0} + \xi_{1} +
\frac{\om}{\be} \log (\ve) + a - \frac{\om}{\be} \log (\la X_{0} +
B (1 + A \sin (\thet_{0}))).
\end{align*}
This proves~\bpref{c1}.

Restricting $\subasup{F}{a,0}{2}$ to the circle $\{ (X_{0},
\thet_{0}): X_{0} = 0 \}$, we obtain the one-parameter family of
circle maps
\begin{equation*}
f_{a} (\thet) = \thet + \xi_{1} + \frac{\om}{\be} \log (\ve) + a -
\frac{\om}{\be} \log (B (1 + A \sin (\thet))).
\end{equation*}
It follows directly from Proposition~\ref{prop2-s4.2} that $f_a$ is an admissible family
of 1D maps provided $\omega \beta^{-1}$ is sufficiently large. This proves~\bpref{c2}.
Hypotheses~\bpref{c3} and~\bpref{c4} follow from direct computation.

We have shown that the family $\{ F_{a, b_{n}} \}$ is an admissible rank one family and
therefore Propositions~\ref{prop3-s3.2} and ~\ref{prop4-s3.2} apply. We conclude that if
$|\om|$ is sufficiently large, then there exists a set $\De_{\om}$ of positive Lebesgue
measure such that for $\mu \in \De_{\om}$, $\mscr{F}_{\mu}$ admits a strange attractor on
$\Si^{-}$ with an ergodic SRB measure $\nu$ and Lebesgue almost every point on $\Si^-$ is
generic with respect to $\nu$. Furthermore, the set $\De_{\om}$ has positive lower
Lebesgue density at $0$, meaning that
\begin{equation*}
\varliminf_{s \to 0^{+}} \frac{|\De_{\om} \cap [0,s]|}{s} > 0
\end{equation*}
where $| \cdot |$ denotes Lebesgue measure.

\section{Standard forms around the homoclinic loop}\label{s2}

In this section we introduce a sequence of coordinate changes to
transform equation ~(\ref{f3-s1.1}) into certain standard forms.
In Section~\ref{s2.1} we work in a sufficiently small neighborhood
$U_{\varepsilon}$ of $(0, 0)$ in the $(x, y)$-plane. In
Section~\ref{s3.1} we work in a small neighborhood around the
entire length of the homoclinic loop $\ell$ outside of
$U_{\frac{1}{4}\varepsilon^2}$. In Section~\ref{s2.3} we define
the Poincar\'e sections $\Sigma^{\pm}$ which we will use to
compute the flow-induced maps. Points on $\Sigma^{\pm}$ are
represented differently by various sets of variables introduced in
Sections~\ref{s2.1} and~\ref{s3.1}. We discuss the issue of
coordinate conversion in Section~\ref{s2.3}.

In the rest of this paper, $\alpha$, $\beta$, $\rho \in [\rho_1,
\rho_2]$ and $\omega > \omega_0$ (it suffices to assume $\om > 0$)
are all regarded as fixed
constants. The size of the neighborhood on which all of the
coordinate transformations in Section~\ref{s2.1} are performed is
determined by a small number $\varepsilon > 0$. The quantity
$\varepsilon$ is also regarded as a fixed constant. We regard
$\mu$ as the only parameter of equation ~(\ref{f3-s1.1}).

\vspace{0.1cm}

\noindent {\bfseries Two small scales.} The quantities $\mu \ll
\varepsilon \ll 1$ represent two small scales of different magnitude.
The quantity $\varepsilon$ represents the size of a small neighborhood of $(x,
y) = (0, 0)$ in which the local analysis of Section~\ref{s2.1} is
valid. Define
\begin{equation*}
  U_{\varepsilon} = \{ (x, y): x^2 + y^2 < 4 \varepsilon^2 \} \text{ and }
  \mscr{U}_{\varepsilon} = U_{\varepsilon} \times \mbb{S}^{1}.
\end{equation*}
Let $L^+$ and $-L^-$ be the respective times at which the homoclinic solution $\ell(t)$
enters $U_{\frac{1}{2}\varepsilon}$ in the positive and negative directions.  The
quantities $L^+$ and $L^-$ are completely determined by $\varepsilon$ and $\ell$. The
parameter $\mu$ ($\mu \ll \varepsilon$) controls the magnitude of the time-periodic
perturbation.

\vspace{0.1cm}

\noindent {\bfseries Notation.} Quantities that are independent of
phase variables, time and $\mu$ are regarded as constants and $K$
is used to denote a generic constant, the precise value of which
is allowed to change from line to line. On occasion, a specific
constant is used in different places. We use subscripts to denote
such constants as $K_0, K_1, \cdots$. We will also distinguish
between constants that depend on $\varepsilon$ and those
that do not by making such dependencies explicit. A constant that
depends on $\varepsilon$ is written as $K(\varepsilon)$. A
constant written as $K$ is independent of $\varepsilon$.

\subsection{Standard form near the fixed point} \label{s2.1}

In this subsection we study equation ~(\ref{f3-s1.1}) in a
sufficiently small neighborhood of $(0, 0)$ in the $(x, y)$-plane.
We introduce a sequence of coordinate changes to transform
equation ~(\ref{f3-s1.1}) into a certain standard form.
Table~\ref{ta:trans} summarizes the purpose of each coordinate transformation.

\begin{center}
\begin{table}[ht]
\caption{Transformations near the fixed point.}
\label{ta:trans}
\begin{tabular}{|c|c|}
\hline
Transformation & Purpose\\
\hline
$(x,y) \to (\xi, \et)$ & linearize the flow defined by~\pref{f1-s1.1} in a neighborhood of $(0,0)$\\
\hline
$(\xi, \et) \to (X,Y)$ & standardize the location of the hyperbolic periodic orbit\\
\hline
$(X,Y) \to (\mathbf{X}, \mathbf{Y})$ & flatten the local invariant manifolds \\
\hline
$(\mathbf{X}, \mathbf{Y}) \to (\mbb{X}, \mbb{Y})$ & rescale by the factor $\mu^{-1}$\\
\hline
\end{tabular}
\end{table}
\end{center}

\vspace{0.2cm}

\refstepcounter{forms} \label{sss:forma} \noindent {\bfseries
\sffamily \ref{sss:forma}. First coordinate change:} $(x, y) \to
(\xi, \eta)$. \ Let $(\xi, \eta)$ be such that
\begin{equation}\label{f1-s2.1a}
\xi = x + q_1(x, y), \quad \eta = y + q_2(x, y)
\end{equation}
where $q_1(x, y)$ and $q_2(x, y)$ are analytic terms of order at least
two in $x$ and $y$. Formula ~(\ref{f1-s2.1a}) defines a
near-identity coordinate transformation $(x, y) \to (\xi, \eta)$,
the inverse of which we write as
\begin{equation}\label{f2-s2.1a}
x = \xi + Q_1(\xi, \eta), \quad y = \eta + Q_2(\xi, \eta).
\end{equation}
\begin{proposition}\label{prop-2.1}
Assume that $\alpha$ and $\beta$ satisfy the
nonresonance condition~\bpref{li:h1}\bpref{li:h1a}.  Then there exists a neighborhood
$U$ of $(0, 0)$, the size of which is completely determined by
equation ~(\ref{f1-s1.1}) and $d_1$ and $d_2$ in~\bpref{li:h1}\bpref{li:h1a}, such that on
$U$ there exists an analytic coordinate transformation
~(\ref{f1-s2.1a}) that transforms equation ~(\ref{f1-s1.1}) into the
linear system
\begin{equation*}
\frac{d \xi}{dt} = - \alpha \xi, \quad \frac{d \eta}{dt} = \beta
\eta.
\end{equation*}
\end{proposition}

\begin{proof}
See~\cite{HmPcSm1977} for a proof.
\end{proof}

We now use the coordinate transformation of Proposition
~\ref{prop-2.1} to transform equation ~(\ref{f3-s1.1}). Observe that
by definition, $q_1(x, y)$ and $q_2(x, y)$ satisfy
\begin{gather}
\sublabon{equation}
(1 + \partial_x q_1(x, y))(-\alpha x + f(x, y)) + \partial_y
q_1(x, y)(\beta y + g(x, y)) = -\alpha \xi
\label{e:q1}\\
(1 + \partial_y q_2(x, y))(\beta y + g(x, y)) + \partial_x q_2(x,
y)(-\alpha x + f(x, y)) = \beta \eta. \label{e:q2}
\end{gather}
\sublaboff{equation}
We derive the form of ~(\ref{f3-s1.1}) in terms of $\xi$ and $\eta$. We
have
\begin{align*}
\frac{d \xi}{dt} &= (1 + \partial_x q_1(x, y))(-\alpha x + f(x,
y) - \mu (\rho h(x, y)+ \sin \theta) )\\
&\qquad {}+ \partial_y q_1(x, y) (\beta y + g(x, y) + \mu (\rho h(x, y)+ \sin \theta)) \\
&= -\alpha \xi - \mu (1 + \partial_x q_1(x, y)- \partial_y
q_1(x, y) ) (\rho h(x, y)+\sin \theta)
\end{align*}
where~\pref{e:q1} is used for the second
equality. Similarly, we have
\begin{align*}
\frac{d \eta}{dt} &= (1 + \partial_y q_2(x, y))(\beta y + g(x,
y) + \mu (\rho h(x, y)+ \sin \theta)) \\
&\qquad {}+ \partial_x q_2(x, y) (-\alpha x + f(x, y)
- \mu (\rho h(x, y)+ \sin \theta)) \\
&= \beta \eta  + \mu(1+ \partial_y q_2(x, y) - \partial_x q_2(x,
y))(\rho h(x, y)+ \sin \theta).
\end{align*}
Writing the functions of $x$ and $y$ as functions of $\xi$ and $\eta$ using
~(\ref{f2-s2.1a}), the form of ~(\ref{f3-s1.1}) in terms of $\xi$
and $\eta$ is given by
\begin{equation}\label{f4-s2.1a}
\left\{
\begin{aligned}
\frac{d\xi}{dt} &= - \alpha \xi - \mu (1+ h_1(\xi,
\eta)) (\rho H(\xi, \eta)+ \sin \theta) \\
\frac{d\eta}{dt} &= \beta \eta  + \mu (1 + h_2(\xi,
\eta))(\rho H(\xi, \eta)+ \sin \theta) \\
\frac{d \theta}{dt} &= \omega
\end{aligned}
\right.
\end{equation}
where $h_1(\xi, \eta) = \partial_x q_1(x, y) - \partial_y q_1(x,
y)$, $h_2(\xi, \eta) =
\partial_y q_2(x, y) - \partial_x q_2(x, y)$ are such that $h_1(0, 0) = h_2(0, 0) = 0$ and
$H(\xi, \eta) = h(x, y)$.

\vspace{0.2cm}

\refstepcounter{forms} \label{sss:formb} \noindent {\bfseries
\sffamily \ref{sss:formb}.  Second coordinate change:} $(\xi,
\eta) \to (X, Y)$. \ With the forcing added, the hyperbolic fixed
point $(x, y) = (0, 0)$ of equation ~(\ref{f1-s1.1}) is perturbed
to become a hyperbolic periodic solution of~(\ref{f2-s1.1}) period $2 \pi
\omega^{-1}$. We denote this periodic solution
in $(\xi, \eta, \theta)$-coordinates as $\xi = \mu \phi(\theta;
\mu)$, $\eta = \mu \psi(\theta; \mu)$.
\begin{proposition}\label{prop-2.2}
For equation ~(\ref{f4-s2.1a}), there exists a unique solution of
the form
\begin{equation*}
\xi = \mu \phi(\theta; \mu), \quad \eta = \mu \psi(\theta; \mu), \quad
\theta = \omega t
\end{equation*}
satisfying
\begin{equation*}
\phi(\theta; \mu) = \phi(\theta + 2 \pi; \mu), \quad \psi(\theta;
\mu) = \psi(\theta + 2 \pi; \mu).
\end{equation*}
The $C^3$ norms of the functions $\phi(\theta; \mu)$ and
$\psi(\theta; \mu)$, regarded as functions of $\theta$ and $\mu$,
are bounded by a constant $K$.
\end{proposition}

\begin{proof}
Write $\phi = \phi(\theta; \mu)$, $\psi =
\psi(\theta; \mu)$. The functions $\phi$ and $\psi$ should satisfy
\begin{equation}\label{f2-s2.1b}
\begin{aligned}
\omega \frac{d \phi}{d \theta} &= - \alpha \phi - (1 + h_1(\mu
\phi, \mu \psi)) (\rho H(\mu \phi, \mu \psi)+ \sin \theta) \\
\omega \frac{d \psi}{d \theta} &= \beta \psi  +(1 +  h_2(\mu
\phi, \mu \psi)) (\rho H(\mu \phi, \mu \psi)+ \sin \theta).
\end{aligned}
\end{equation}
From ~(\ref{f2-s2.1b}) it follows that
\begin{align*}
\phi(\theta; \mu) &= e^{- \alpha \omega^{-1} (\theta - \theta_0)}
\phi(\theta_0; \mu) -\omega^{-1} \int_{\theta_0}^{\theta}
e^{\alpha\omega^{-1} (s - \theta)} [1 + h_1(\mu \phi(s; \mu),
\mu \psi(s; \mu))] \times \\
&\qquad [\rho H(\mu \phi(s; \mu), \mu \psi(s; \mu))+
\sin s] \, ds \\
\psi(\theta; \mu) &= e^{\beta \omega^{-1} (\theta - \theta_0)}
\psi(\theta_0; \mu) + \omega^{-1} \int_{\theta_0}^{\theta}
e^{-\beta \omega^{-1}(s - \theta)}[1+ h_2(\mu \phi(s; \mu), \mu
\psi(s; \mu))] \times \\
&\qquad [\rho H(\mu \phi(s; \mu), \mu \psi(s; \mu))+ \sin s] \, ds.
\end{align*}
To solve for $\phi$ and $\psi$ we let $\theta = \theta_0 + 2 \pi$
and set $\phi(\theta_0 + 2 \pi; \mu) = \phi(\theta_0; \mu)$,
$\psi(\theta_0 + 2 \pi; \mu) = \psi(\theta_0; \mu)$, obtaining
\begin{equation}\label{f3-s2.1b}
\begin{aligned}
\phi(\theta; \mu) &= \frac{- \omega^{-1}}{1 - e^{- 2 \alpha
\omega^{-1} \pi}} \int_0^{2 \pi} e^{\alpha\omega^{-1} (s - 2 \pi)}
[1 + h_1(\mu \phi(s+\theta; \mu),
\mu \psi(s+\theta; \mu))] \times \\
&\qquad [\rho H(\mu \phi(s+ \theta; \mu), \mu \psi(s+\theta;
\mu))+ \sin (s+\theta)] \, ds \\
\psi(\theta; \mu) &= \frac{\omega^{-1}}{1-e^{2 \beta \omega^{-1}
\pi}} \int_{0}^{2 \pi} e^{-\beta \omega^{-1}(s - 2 \pi)}[1+
h_2(\mu \phi(s+\theta; \mu), \mu
\psi(s+ \theta; \mu))] \times \\
&\qquad [\rho H(\mu \phi(s+\theta; \mu), \mu \psi(s+\theta; \mu))+
\sin (s+ \theta)] \, ds.
\end{aligned}
\end{equation}
The existence and uniqueness of $\phi(\theta; \mu)$ and
$\psi(\theta; \mu)$ follows directly from an application of the
contraction mapping theorem to ~(\ref{f3-s2.1b}). The asserted
bound on partial derivatives with respect to $\theta$ and $\mu$
follows from differentiating ~(\ref{f3-s2.1b}) with respect to
$\theta$ and $\mu$.
\end{proof}

We now introduce new variables $(X, Y)$ by defining
\begin{equation}\label{f1-s2.1b}
X  = \xi - \mu \phi(\theta; \mu), \quad Y = \eta - \mu
\psi(\theta; \mu).
\end{equation}
We have
\begin{align*}
\frac{dX}{dt} &= -\alpha X - \alpha \mu \phi - \mu \omega \frac{d
\phi}{d \theta} - \mu (1+ h_1(X + \mu \phi, Y + \mu \psi))(\rho H(X+\mu
\phi, Y + \mu \psi)+ \sin \theta)
\\
\frac{dY}{dt} &= \beta Y + \beta \mu \psi - \mu \omega \frac{d
\psi}{d \theta} + \mu (1 + h_2(X + \mu \phi, Y + \mu \psi))(\rho H(X + \mu
\phi, Y + \mu \psi)+ \sin \theta).
\end{align*}
Using ~(\ref{f2-s2.1b}), the form of ~(\ref{f3-s1.1}) in terms of
$X$, $Y$ and $\theta$ is given by
\begin{equation}\label{f5-s2.1b}
\left\{
\begin{aligned}
\frac{dX}{dt} &= -\alpha X + \mu F(X, Y, \theta; \mu)    \\
\frac{dY}{dt} &= \beta Y + \mu G(X, Y, \theta; \mu) \\
\frac{d \theta}{dt} &= \omega
\end{aligned}
\right.
\end{equation}
where
\begin{align*}
F(X, Y, \theta; \mu) &= -[h_1 (X + \mu \phi, Y + \mu \psi) -
h_1(\mu \phi, \mu \psi)] (\rho H(X + \mu \phi, Y + \mu \psi) + \sin \theta)\\
&\qquad {}- \rho (1+h_1(\mu \phi,
\mu \psi)) (H(X + \mu \phi, Y+
\mu \psi) - H(\mu \phi, \mu \psi)) \\
G(X, Y, \theta; \mu) &= [h_2 (X + \mu \phi, Y + \mu \psi) -
h_2(\mu
\phi, \mu \psi)](\rho H(X+\mu \phi, Y + \mu \psi) + \sin \theta)\\
&\qquad {}+ \rho(1+ h_2(\mu \phi,
\mu \psi))(H(X + \mu \phi, Y+ \mu \psi) - H(\mu \phi, \mu \psi))
\end{align*}
are such that $F(0, 0, \theta; \mu) = G(0, 0, \theta; \mu) = 0$.
Observe that in the new coordinates $(X, Y, \theta)$, the solution
$\xi = \mu \phi(\theta; \mu)$, $\eta = \mu \psi(\theta;\mu)$ is
represented by $X = Y = 0$. We remark that on
\begin{equation*}
\{(X, Y, \theta; \mu) : \| (X, Y) \| < \varepsilon, \; \theta \in
\mbb{S}^{1}, \; 0 \leqs \mu \leqs \mu_0 \},
\end{equation*}
\begin{enumerate}
\item
$F(X, Y, \theta; \mu)$ and $G(X, Y, \theta; \mu)$ are analytic
functions bounded by $K \varepsilon$;
\item
it follows from Proposition ~\ref{prop-2.2} that the
$C^3$ norms of both $F$ and $G$ as functions of $(X, Y, \theta)$
and $\mu$ are bounded by a constant $K$.
\end{enumerate}

\vspace{0.2cm}

\refstepcounter{forms} \label{sss:formc} \noindent {\bfseries
\sffamily \ref{sss:formc}. Third coordinate change:} $(X, Y) \to
({\bf X, \bf Y})$. \ The periodic solution $(X, Y, \theta) = (0,
0, \omega t)$ of equation ~(\ref{f5-s2.1b}) has a local unstable
manifold, which we write as
\begin{equation*}
X = \mu W^u(Y, \theta; \mu),
\end{equation*}
and a local stable manifold, which we write as
\begin{equation*}
Y = \mu W^s(X, \theta; \mu).
\end{equation*}

\begin{proposition}\label{prop-2.3}
There exists $\varepsilon > 0$ and $\mu_0 = \mu_0(\varepsilon)
> 0$ such that $W^u(Y, \theta; \mu)$ and $W^s(X,
\theta; \mu)$ are analytically defined on
\begin{equation*}
(-\varepsilon, \varepsilon) \times \mbb{S}^{1} \times  [0, \mu_0]
\end{equation*}
and satisfy
\begin{equation*}
W^u(0, \theta; \mu)=0, \quad W^s(0, \theta; \mu) = 0.
\end{equation*}
The $C^3$ norms of $W^u(Y, \theta; \mu)$ and $W^s(X, \theta; \mu)$,
regarded as functions of all three of their arguments, are
bounded by a constant $K$.
\end{proposition}

\begin{proof}
We regard $X$, $Y$, $\theta$, and $\mu$ in
equation ~(\ref{f5-s2.1b}) as complex variables. The existence and
smoothness of local stable and unstable manifolds follows from the
standard argument based on the contraction mapping theorem. See
~\cite{HmPcSm1977} for instance.
\end{proof}

By definition, $W^u(Y, \theta; \mu)$ satisfies
\begin{equation}\label{f1-s2.1c}
\begin{aligned}
-\alpha W^u &(Y, \theta; \mu) + F(\mu W^u(Y, \theta; \mu), Y,
\theta; \mu) =
\omega \partial_{\theta} W^u(Y, \theta; \mu) \\
&+ \partial_Y W^u(Y, \theta; \mu) (\beta Y + \mu
G(\mu W^u(Y, \theta; \mu), Y, \theta; \mu)).
\end{aligned}
\end{equation}
Similarly, $W^s(X, \theta; \mu)$ satisfies
\begin{equation}\label{f2-s2.1c}
\begin{aligned}
\beta W^s &(X, \theta; \mu) + G(X, \mu W^s(X, \theta; \mu),
\theta; \mu) =
\omega \partial_{\theta} W^s(X, \theta; \mu) \\
&+ \partial_X W^s(X, \theta; \mu) (-\alpha X + \mu
F(X, \mu W^s(X, \theta; \mu), \theta; \mu)).
\end{aligned}
\end{equation}

Define the new variables $\mathbf{X}$ and $\mathbf{Y}$ by
\begin{equation}\label{f2a-s2.1c}
{\bf X} = X - \mu W^u(Y, \theta; \mu), \quad {\bf Y} = Y - \mu
W^s(X, \theta; \mu).
\end{equation}
By using ~(\ref{f5-s2.1b}), ~(\ref{f1-s2.1c}), and ~(\ref{f2-s2.1c}),
the form of ~(\ref{f3-s1.1}) in terms of $({\bf X}, {\bf Y},
\theta)$ is given by
\begin{equation}\label{f3-s2.1c}
\left\{
\begin{aligned}
\frac{d {\bf X}}{d t}  &= (-\alpha + \mu {\bf F}({\bf X}, {\bf Y}, \theta; \mu)) {\bf X} \\
\frac{d {\bf Y}}{d t}  &= (\beta + \mu {\bf G}({\bf X}, {\bf Y},
\theta; \mu)) {\bf
Y}\\
\frac{d \theta}{dt} &= \omega.
\end{aligned}
\right.
\end{equation}
where ${\bf F}$ and ${\bf G}$ are analytic functions of ${\bf X}$, ${\bf
Y}$, $\theta$, and $\mu$ defined on $U_{\varepsilon} \times \mbb{S}^{1}
\times [0, \mu_0]$.  The $C^3$ norms of ${\bf F}$ and ${\bf G}$
are bounded by a constant $K$.
Tracing back to the variables $(\xi, \eta)$, we have
\begin{align}
\sublabon{equation}
{\bf X} &= \xi - \mu \left( \phi(\theta; \mu) +
W^u(\eta - \mu \psi(\theta; \mu), \theta; \mu)\right)
\label{e:convertX}\\
{\bf Y} &= \eta - \mu \left(\psi(\theta; \mu) + W^s(\xi - \mu
\phi(\theta; \mu), \theta; \mu)\right).
\label{e:convertY}
\end{align}
\sublaboff{equation}

\vspace{0.2cm}

\refstepcounter{forms} \label{sss:formd} \noindent {\bfseries
\sffamily \ref{sss:formd}.  Fourth coordinate change:} $({\bf X},
{\bf Y}) \to ({\mathbb X}, {\mathbb Y})$. \ The final coordinate
change is a rescaling of ${\bf X}$ and ${\bf Y}$ by the factor
$\mu^{-1}$. Let
\begin{equation}\label{f1a-s2.1d}
{\mathbb X} = \mu^{-1} {\bf X}, \quad {\mathbb Y} = \mu^{-1} {\bf
Y}.
\end{equation}
We write equation ~(\ref{f3-s2.1c}) in ${\mathbb X}$ and ${\mathbb Y}$
as
\begin{equation}\label{f1-s2.1d}
\left\{
\begin{aligned}
\frac{d {\mathbb X}}{d t}  &= (-\alpha + \mu {\mathbb F}({\mathbb
X}, {\mathbb Y}, \theta;
\mu)) {\mathbb X} \\
\frac{d {\mathbb Y}}{d t}  &= (\beta + \mu {\mathbb G}({\mathbb
X}, {\mathbb Y}, \theta; \mu)) {\mathbb
Y}\\
\frac{d \theta}{dt} &= \omega
\end{aligned}
\right.
\end{equation}
where
\begin{equation*}
{\mathbb F}({\mathbb X}, {\mathbb Y}, \theta; \mu)  = {\bf F}(\mu
{\mathbb X}, \mu {\mathbb Y}, \theta; \mu), \quad {\mathbb
G}({\mathbb X}, {\mathbb Y}, \theta; \mu)  = {\bf G}(\mu {\mathbb
X}, \mu {\mathbb Y}, \theta; \mu)
\end{equation*}
are analytic functions of ${\mathbb X}$, ${\mathbb Y}$, $\theta$, and
$\mu$ defined on
\begin{equation*}
{\mathbb D} = \{ ({\mathbb X}, {\mathbb Y}, \theta, \mu) : \mu
\in [0, \mu_0], \; ({\mathbb X}, {\mathbb Y}, \theta) \in
\mscr{U}_{\varepsilon} \}
\end{equation*}
where
\begin{equation*}
\mscr{U}_{\varepsilon} = \{ ({\mathbb X}, {\mathbb Y},
\theta): \|({\mathbb X}, {\mathbb Y})\| < 2 \varepsilon
\mu^{-1}, \; \theta \in \mbb{S}^{1} \}.
\end{equation*}

\begin{remark}
We remind the reader that all constants
represented by $K$ in Section~\ref{s2.1} are independent of
$\varepsilon$ and $\mu$.
\end{remark}

\subsection{A standard form around the homoclinic loop}\label{s3.1}
In this subsection we derive a standard form for equation
~(\ref{f3-s1.1}) around the homoclinic loop of equation
~(\ref{f1-s1.1}) outside of ${\mathcal
U}_{\frac{1}{4}\varepsilon^2}$. Some elementary estimates are also
included.

\vspace{0.2cm}

\refstepcounter{derivaeq} \label{sss:derivationeq} \noindent
{\bfseries \sffamily \ref{sss:derivationeq}. Derivation of
equations.} \ Let us regard $t$ in $\ell(t) =(a(t), b(t))$ not as
time, but as a parameter that parametrizes the curve $\ell$ in
$(x, y)$-space. We replace $t$ by $s$ and write this homoclinic
loop as $\ell(s) = (a(s), b(s))$. We have
\begin{equation}\label{f1-s3.1}
\begin{aligned}
\frac{d a(s)}{ds} &= -\alpha a(s) +  f(a(s), b(s))\\
\frac{d b(s)}{ds} &= \beta b(s) + g(a(s), b(s)).
\end{aligned}
\end{equation}
Define
\begin{equation*}
(u(s), v(s)) = \left\| \frac{d}{ds} \ell (s) \right\|^{-1} \frac{d}{ds} \ell (s).
\end{equation*}
We have
\begin{equation}\label{f2-s3.1}
\begin{aligned}
u(s) &= \frac{-\alpha a(s) + f(a(s), b(s))}{\sqrt{(-\alpha a(s) +
f(a(s), b(s)))^2 + (\beta b(s) + g(a(s)
b(s)))^2}}, \\
v(s) &= \frac{\beta b(s) + g(a(s), b(s))}{\sqrt{(-\alpha a(s) +
f(a(s), b(s)))^2 + (\beta b(s) +  g(a(s), b(s)))^2}}.
\end{aligned}
\end{equation}
Let
\begin{equation*}
\mbf{e} (s) = (v(s), -u(s)).
\end{equation*}
The vector $\mbf{e} (s)$ is the inward unit normal vector to $\ell$ at $\ell (s)$.  We now
introduce the new variable $z$ such that
\begin{equation*}
(x, y) = \ell(s) + z \mbf{e} (s).
\end{equation*}
That is,
\begin{equation}\label{f3-s3.1}
x = x(s, z) = a(s) + v(s) z, \quad y =y(s, z)= b(s) - u(s) z.
\end{equation}

We derive the form of ~(\ref{f3-s1.1}) in terms of the new
variables $(s, z)$ defined through ~(\ref{f3-s3.1}).
Differentiating ~(\ref{f3-s3.1}), we obtain
\begin{equation}\label{f4-s3.1}
\begin{aligned}
\frac{dx}{dt} &= (- \alpha a(s) + f(a(s), b(s)) + v'(s) z)
\frac{ds}{dt} + v(s)
\frac{dz}{dt} \\
\frac{dy}{dt} &= (\beta b(s) + g(a(s), b(s)) - u'(s) z)
\frac{ds}{dt} - u(s) \frac{dz}{dt}
\end{aligned}
\end{equation}
where $u'(s) = \frac{d u(s)}{ds}$ and $v'(s) = \frac{d v(s)}{ds}$.
Denote
\begin{align*}
F(s, z) &= - \alpha (a(s) + z v(s)) + f(a(s) + z v(s), b(s) - z u(s)) \\
G(s, z) &= \beta (b(s) - z u(s)) + g(a(s) + z v(s), b(s) - z
u(s)) \\
{\mathbb H}(s, z) &= h(a(s) + z v(s), b(s) - z u(s)).
\end{align*}
Using ~(\ref{f3-s1.1}) and ~(\ref{f4-s3.1}), we have
\begin{align*}
\frac{ds}{dt} &= \frac{v(s) G(s, z) + u(s) F(s, z) + \mu (v(s) -
u(s)) (\rho {\mathbb H}(s, z)+ \sin \theta)}{\sqrt{F(s, 0)^2 +
G(s, 0)^2} + z(u(s)
v'(s) - v(s) u'(s))} \\
\frac{dz}{dt} &= v(s)F(s, z)- u(s) G(s, z) - \mu (u(s)+v(s))
(\rho {\mathbb H}(s, z)+ \sin \theta).
\end{align*}
We rewrite these equations as
\begin{equation}\label{f5-s3.1}
\begin{aligned}
\frac{ds}{dt} &= 1 + z w_1(s, z, \theta; \mu) +
\frac{\mu (v(s) - u(s))(\rho {\mathbb H}(s, 0)+ \sin \theta)}{\sqrt{F(s, 0)^2 + G(s, 0)^2}} \\
\frac{dz}{dt} &= E(s) z + z^2 w_2(s, z) - \mu (u(s) + v(s))
(\rho {\mathbb H}(s, z) +\sin \theta)\\
\frac{d \theta}{dt} &= \omega
\end{aligned}
\end{equation}
where
\begin{align*}
E(s) &= v^2 (s) (-\alpha + \partial_x f(a(s), b(s))) + u^2(s)
(\beta + \partial_y g(a(s), b(s))) \\
&\qquad {}- u(s) v(s) (\partial_y f(a(s), b(s)) + \partial_x g(a(s),
b(s))) \\
{\mathbb H}(s, 0) &= h(a(s), b(s)).
\end{align*}
Equation ~(\ref{f5-s3.1}) is defined on
\begin{equation*}
\{ s \in [-2 L^-, 2 L^+], \; \mu \in [0, \mu_0], \; \theta \in
\mbb{S}^{1}, \; |z| < K_0(\varepsilon) \mu \},
\end{equation*}
where $K_0(\varepsilon)$ is independent of $\mu$. The $C^3$ norms
of the functions $w_1(s, z, \theta; \mu)$ and $w_2(s, z)$ are bounded
by a constant $K(\varepsilon)$.

Finally, we rescale the variable $z$ by letting
\begin{equation}\label{f6-s3.1}
Z = \mu^{-1} z.
\end{equation}
We arrive at the equations
\begin{align}
\sublabon{equation}
\frac{ds}{dt} &= 1 + \mu \tilde w_1(s, Z, \theta; \mu)
\label{e:gnfa}\\
\frac{dZ}{dt} &= E(s) Z + \mu \tilde w_2(s, Z, \theta; \mu)
- (u(s) + v(s)) (\rho {\mathbb H}(s, 0)+ \sin \theta)
\label{e:gnfb}\\
\frac{d\theta}{dt} &= \omega
\label{e:gnfc}
\end{align}
\sublaboff{equation}
defined on
\begin{equation*} 
{\bf D} = \{ (s, Z, \theta; \mu) : s \in [-2L^-, 2L^+], \; |Z| \leqs K_{0} (\varepsilon),
\; \theta \in \mbb{S}^{1}, \; \mu \in [0, \mu_0] \}.
\end{equation*}
We assume that $\mu_0$ is sufficiently small so that
\begin{equation*}
\mu \ll \min_{s \in [-2L^-, 2L^+]} (F(s, 0)^2 + G(s, 0)^2).
\end{equation*}
The $C^3$ norms of the functions $\tilde w_1$ and $\tilde w_2$ are
bounded by a constant $K(\varepsilon)$ on ${\bf D}$.

System~\pref{e:gnfa}--\pref{e:gnfc} is the one we need. The function $E(s)$
appears in the integrals $A$, $C$, and $S$ in~\bpref{li:h2}.

\begin{remark}
Observe that all of the generic
constants that have appeared thus far in this subsection have the
form $K(\varepsilon)$.
\end{remark}

\vspace{0.2cm} \refstepcounter{derivaeq} \label{sss:techest}
\noindent  {\bfseries \sffamily \ref{sss:techest}. Technical
estimates.} \ We adopt the following conventions in comparing the
magnitude of two functions $f(s)$ and $g(s)$. We write $f(s) \prec
g(s)$ if there exists $K > 0$ independent of $s$ such that $
|f(s)|< K |g(s)|$ as $s \to \infty$ (or $-\infty$). We write $f(s)
\sim g(s)$ if in addition we have $|f(s)| > K^{-1} |g(s)|$. We
also write $f(s) \approx g(s)$ if
\begin{equation*}
\frac{f(s)}{g(s)} \to 1
\end{equation*}
as $s \to \infty$ (or $-\infty$).

Recall that $\ell(s) = (a(s), b(s))$ is the homoclinic solution
for the hyperbolic fixed point $(0, 0)$ of equation
~(\ref{f1-s1.1}). The vector $(u(s), v(s))$ is the unit tangent
vector to $\ell$ at $\ell(s)$.

\begin{lemma} \label{lem1-s2.2b}
As $s \to +\infty$, we have
\begin{enumerate}
\item $a(s) \sim e^{-\alpha s}$, $a(-s) \prec e^{- 2\beta s}$
\item $b(s) \prec e^{-2 \alpha s}$, $b(-s) \sim e^{- \beta s}$
\item $u(s) \approx -1$, $u(-s) \prec e^{-\beta s}$ \item $v(s)
\prec e^{-\alpha s}$, $v(-s) \approx 1$.
\end{enumerate}
\end{lemma}

\begin{proof}
We are simply restating the fact that
$\ell(s) \to (0, 0)$ with an exponential rate $-\alpha$ in the
positive $s$-direction along the $x$-axis and with an exponential
rate $\beta$ in the negative $s$-direction along the $y$-axis.
\end{proof}

\begin{lemma}\label{lem2-s2.2b}
Let $E(s)$ be as in ~(\ref{f1a-s1.1}).  As $L^{\pm} \to +\infty$,
we have
\begin{enumerate}[(a)]
\item \label{li:Einta} $\int_{-L^-}^0 (E(s) + \alpha) \, ds \prec
1$ \item \label{li:Eintb} $\int_0^{L^+} (E(s) - \beta) \, ds \prec
1$ \item \label{li:Eintc} $\int_{-L^-}^0 E(s) \, ds \approx -
\alpha L^-$ \item \label{li:Eintd} $\int_0^{L^+} E(s) \, ds
\approx \beta L^+$.
\end{enumerate}
\end{lemma}

\begin{proof}
Statements~\pref{li:Einta} and~\pref{li:Eintb} claim that the
integrals are convergent as $L^{\pm} \to \infty$.
For~\pref{li:Einta}, we observe that by adding $\alpha$ to $E(s)$,
we obtain $E(s) + \alpha$ as a collection of terms, each of which
decays exponentially as $s \to -\infty$ by Lemma
~\ref{lem1-s2.2b}. Similarly, taking $\beta$ away from $E(s)$, we
obtain $E(s) - \beta$ as a collection of terms, each of which
decays exponentially as $s \to \infty$.

For~\pref{li:Eintc} and~\pref{li:Eintd} we write
\begin{gather*}
\int_{-L^-}^0 E(s) \, ds = - \alpha L^- + \int_{-L^-}^0 (E(s) +
\alpha)
\, ds \\
\int_{0}^{L^+} E(s) \, ds = \beta L^+  + \int_{0}^{L^+} (E(s)-\beta) \, ds. \\
\end{gather*}
Statements~\pref{li:Eintc} and~\pref{li:Eintd} now follow from~\pref{li:Einta} and~\pref{li:Eintb}, respectively.
\end{proof}

\begin{lemma}\label{lem3-s2.2b}
All of the integrals defined in ~(\ref{f2a-s1.1}) are absolutely
convergent.
\end{lemma}

\begin{proof}
Let us write
\begin{equation*}
\begin{split}
A = &\int_{-\infty}^{-L_0}(u(s) + v(s)) h(a(s), b(s))
e^{-\int_{0}^s E(\tau) d \, \tau} \, ds\\
&+ \int_{-L_0}^{L_0}(u(s) + v(s)) h(a(s), b(s))
e^{-\int_{0}^s E(\tau) d \, \tau} \, ds \\
&+ \int_{L_0}^{\infty} (u(s) + v(s)) h(a(s), b(s))
e^{-\int_{0}^s E(\tau) d \, \tau} \, ds.
\end{split}
\end{equation*}
We write the first integral as
\begin{equation*}
\int_{-\infty}^{-L_0} (u(s) + v(s)) h(a(s), b(s)) e^{\alpha s}
e^{-\int_{0}^s (E(\tau)+\alpha) d \, \tau} \, ds
\end{equation*}
and make $L_0$ sufficiently large so that $|E(\tau) + \alpha| <
\frac{1}{2} \alpha$ for all $\tau \in (-\infty, -L_0)$. This
integral is convergent since the integrand is $< K e^{\frac{1}{2}
\alpha s}$ for all  $s \in (- \infty, -L_0)$. For the convergence
of the third integral we rewrite it as
\begin{equation*}
\int_{L_0}^{\infty} (u(s) + v(s)) h(a(s), b(s)) e^{-\beta s}
e^{-\int_{0}^s (E(\tau) - \beta) d \, \tau} \, ds
\end{equation*}
and observe that $|E(\tau) - \beta| < \frac{\beta}{2}$ for $\tau
\in [L_0, \infty)$ provided that $L_0$ is sufficiently large.
The proofs for $C$ and $S$ are similar.
\end{proof}

\subsection{Poincar\'e sections and conversion of coordinates} \label{s2.3}
In this subsection we introduce the Poincar\'e sections
$\Sigma^{\pm}$. Since various sets of phase variables have
appeared in Sections ~\ref{s2.1} and ~\ref{s3.1}, we also need to
know how to explicitly convert coordinates from one set to another
on $\Sigma^{\pm}$.

\vspace{0.2cm}

\refstepcounter{psect} \label{sss:poincare} \noindent {\bfseries
\sffamily \ref{sss:poincare}. The Poincar\'e Sections} {\mathversion{bold} $\mathsf{\Sigma^{\pm}}$}. \ Recall
that $\{ \ell(s): s \in (-\infty, \infty) \}$ is the homoclinic
loop of equation ~(\ref{f1-s1.1}). Given $\varepsilon
> 0$ sufficiently small, let $L^+$ and $-L^-$ be such that
\begin{equation}\label{f1-s2.3a}
\begin{aligned}
\xi(-L^-) &= a(-L^-) + q_1(a(-L^-), b(-L^-)) = 0 \\
\eta(-L^-) &= b(-L^-) + q_2(a(-L^-), b(-L^-)) = \varepsilon \\
\xi(L^+) &= a(L^+) + q_1(a(L^+), b(L^+))= \varepsilon \\
\eta(L^+) &= b(L^+) + q_2(a(L^+), b(L^+)) = 0
\end{aligned}
\end{equation}
where $\xi$ and $\eta$ are the variables defined through
~(\ref{f1-s2.1a}). Let
\begin{equation*}
\wh{K}_0 = \max_{\substack{\theta \in \mbb{S}^{1}\\ \mu \in [0, \mu_0]}}
\{ |\phi(\theta;
\mu)|, \, |\psi(\theta; \mu)| \}
\end{equation*}
where $\phi(\theta; \mu)$ and $\psi(\theta; \mu)$ are as in
Section~\ref{s2.1}\ref{sss:formb}. We define two sections in
$\mscr{U}_{\varepsilon}$, denoted $\Sigma^-$ and $\Sigma^+$, as
follows.
\begin{equation}\label{f2-s2.3}
\begin{aligned}
\Sigma^- &= \{ (x, y, \theta) : s = -L^-, \; |z| \leqs (\wh{K}_0
+1) \mu, \;
\theta \in \mbb{S}^{1} \} \\
\Sigma^+ &= \{ (x, y, \theta) : s = L^+, \; \frac{1}{10} (-\rho A)
(\wh{K}_0 + 1) e^{\frac{1}{2} \beta L^+} \mu \leqs z \\
& \ \ \ \ \ \ \ \leqs 10 (-\rho A) (\wh{K}_0 + 1) e^{2 \beta L^+}
\mu, \; \theta \in \mbb{S}^{1} \}
\end{aligned}
\end{equation}
where $s$ and $z$ are as in ~(\ref{f3-s3.1}). We construct the
flow-induced map $\mscr{F}_{\mu}$ in two steps.

\begin{enumerate}[(1)]
\item
Starting from $\Sigma^-$, the solutions of equation
~(\ref{f3-s1.1}) move out of $\mscr{U}_{\varepsilon}$,
following the homoclinic loop of equation ~(\ref{f1-s1.1}) to
eventually hit $\Sigma^+$. This defines a flow-induced map from
$\Sigma^-$ to $\Sigma^+$, which we denote as $\mscr{M}:
\Sigma^- \to \Sigma^+$. We will prove that $\mscr{M}(\Sigma^-)
\subset \Sigma ^+$.
\item
Starting from $\Sigma^+$, the solutions of equation
~(\ref{f3-s1.1}) stay inside of $\mscr{U}_{\varepsilon}$,
carrying $\Sigma^+$ into $\Sigma^-$. This map we denote as
$\mscr{N}$.
\end{enumerate}

We define $\mscr{F}_{\mu} = \mscr{N} \circ \mscr{M}$.
Observe that the variables $(s, Z, \theta)$ of Section~\ref{s3.1}
are suitable for computing $\mscr{M}$ and $({\mathbb X},
{\mathbb Y}, \theta)$ are suitable for computing $\mscr{N}$.
To properly compose $\mscr{N}$ and $\mscr{M}$, we
need to know how to convert from $(s, Z, \theta)$ to $({\mathbb
X}, {\mathbb Y}, \theta)$ on $\Sigma^{\pm}$ and vice-versa.

\vspace{0.1cm}

\noindent
{\bfseries The new parameter} $\mbf{p}$. As stated earlier, we
regard $\mu$ as the only parameter of system ~(\ref{f3-s1.1}). We
make a coordinate change on this parameter by letting $p = \ln
\mu$ and we regard $p$, not $\mu$, as our bottom-line parameter. In
other words, we regard $\mu$ as a shorthand for $e^p$ and all
functions of $\mu$ are thought of as functions of $p$.
Observe that $\mu \in (0, \mu_0]$ corresponds to $p \in (-\infty,
\ln \mu_0]$. This is a {\bfseries \itshape very important conceptual point}
because by regarding a function $F(\mu)$ of $\mu$ as a function of
$p$, we have
\begin{equation*}
\partial_p F(\mu) = \mu \partial_{\mu} F(\mu).
\end{equation*}
Therefore, thinking of $F(\mu)$ as a function of $p$ produces a
$C^3$ norm that is completely different from the one obtained by
thinking of $F(\mu)$ as a function of $\mu$.

\begin{notation}
In order to apply the theory of rank
one maps~\cite{WqYls2001, WqYls2008}, we need to control the
$C^3$ norm of $\mscr{F}_{\mu}$. In particular, we must
estimate the $C^3$ norms of certain quantities with respect to
various sets of variables on relevant domains. The derivation of
the flow-induced maps $\{ \mscr{F}_{\mu} \}$ involves a composition
of maps and multiple coordinate changes. To facilitate the
presentation, from this point on we adopt specific conventions for
indicating controls on magnitude. For a given constant, we write
${\mathcal O}(1)$, ${\mathcal O}(\varepsilon)$, or ${\mathcal
O}(\mu)$ to indicate that the magnitude of the constant is bounded
by $K$, $K \varepsilon$, or $K(\varepsilon) \mu$, respectively. For
a function of a set $V$ of variables on a specific domain, we
write ${\mathcal O}_V(1), {\mathcal O}_V(\varepsilon)$ or
${\mathcal O}_V(\mu)$ to indicate that the $C^3$ norm of the
function on the specified domain is bounded by $K$, $K \varepsilon$,
or $K(\varepsilon) \mu$, respectively. We choose to specify the
domain in the surrounding text rather than explicitly involving it
in the notation. For example, $ {\mathcal O}_{{\mathbb X}_0,
{\mathbb Y}_0, \theta, \mu}(\varepsilon)$ represents a function of
${\mathbb X}_0$, ${\mathbb Y}_0$, $\theta$, and $\mu$, the $C^3$ norm of
which is bounded above by $K \varepsilon$ on a domain explicitly
given in the surrounding text. Similarly, ${\mathcal O}_{Z,
\theta, p}(\mu)$ represents a function of $Z$, $\theta$, and $p$, the
$C^3$ norm of which is bounded above by $K(\varepsilon) \mu$.
\end{notation}

\vspace{0.2cm}

\refstepcounter{psect} \label{sss:conv1} \noindent {\bfseries
\sffamily \ref{sss:conv1}.  Conversion on} {\mathversion{bold} $\mathsf{\Sigma^-}$}. \ The section $\Sigma^-$
is defined by $s = -L^-$. A point $q \in \Sigma^-$ is uniquely
determined by a pair $(Z, \theta)$. First we compute the
coordinates ${\mathbb X}$ and ${\mathbb Y}$ for a point given in
$(Z, \theta)$-coordinates on $\Sigma^-$. Recall that $p = \ln
\mu$.
\begin{proposition}\label{prop-s2.3b}
For $\mu \in (0, \mu_0]$ and $(Z, \theta) \in \Sigma^-$, we have
\begin{align*}
{\mathbb X} &= (1 + {\mathcal O}_{\theta, p}(\varepsilon) + \mu
{\mathcal O}_{Z, \theta, p}(1)) Z - {\mathcal O}_{\theta, p}(1)
\\
{\mathbb Y} &= \mu^{-1} \varepsilon +  {\mathcal O}_{Z, \theta,
p}(1).
\end{align*}
\end{proposition}

\begin{proof}
By definition, $s =-L^-$ on $\Sigma^-$. Let $q \in \Sigma^-$ be
represented by $(z, \theta)$.  Using~\pref{f1-s2.3a}, we have
\begin{equation}\label{f1-s2.3b}
\begin{aligned}
a(-L^-) &= Q_1(0, \varepsilon) = {\mathcal O}(\varepsilon^2) \\
b(-L^-) &= \varepsilon + Q_2(0, \varepsilon) = \varepsilon +
{\mathcal O}(\varepsilon^2).
\end{aligned}
\end{equation}
We also have
\begin{equation}\label{f2-s2.3b}
u(-L^-) = {\mathcal O}(\varepsilon), \quad v(-L^-) = 1 - {\mathcal
O}(\varepsilon).
\end{equation}

We compute values of ${\bf X}$ and ${\bf Y}$ for $q$. Using
~(\ref{f1-s2.3a}) and ~(\ref{f1-s2.3b}),
\begin{align*}
\xi &= a(-L^-) + v(-L^-) z + q_1(a(-L^-)+ v(-L^-)z, \ b(-L^-)-u(-L^-) z)\\
&= v(-L^-) z + q_1(a(-L^-)+ v(-L^-)z, \, b(-L^-)-u(-L^-) z) - q_1(a(-L^-), b(-L^-))\\
&= (1 + {\mathcal O}(\varepsilon)+ z h_{\xi}(z)) z.
\end{align*}
Similarly, we have
\begin{align*}
\eta &= b(-L^-) - u(-L^-) z + q_2(a(-L^-)+ v(-L^-) z, \ b(-L^-)-u(-L^-) z)\\
&= \varepsilon - u(-L^-) z + q_2(a(-L^-)+ v(-L^-) z, \, b(-L^-)-u(-L^-) z)
- q_2(a(-L^-), b(-L^-))\\
&= \varepsilon  + ({\mathcal O}(\varepsilon)+z h_{\eta}(z)) z.
\end{align*}
The functions $h_{\xi}$ and $h_{\eta}$ are analytic on
$|z| < (\wh{K}_0 + 1)\mu$ and we have $h_{\xi}(z) = {\mathcal O}_{z} (1)$
and $h_{\eta}(z) = {\mathcal O}_z(1)$. Substituting $\xi$ and $\eta$
above into~\pref{e:convertX}, we obtain
\begin{align*}
{\bf X} &= (1 + {\mathcal O}(\varepsilon)+ z h_{\xi}(z)) z - \mu
\phi(\theta; \mu) - \mu W^u(\varepsilon - \mu \psi(\theta;
\mu) + ({\mathcal O}(\varepsilon)+z h_{\eta}(z)) z, \theta; \mu) \\
&= (1 + {\mathcal O}(\varepsilon)+ z h_{\xi}(z)) z - \mu
\phi(\theta; \mu) - \mu W^u(\varepsilon
- \mu \psi(\theta; \mu), \theta; \mu) \\
&\qquad {}- \mu W^u(\varepsilon - \mu \psi(\theta; \mu) +
({\mathcal O}(\varepsilon)+z h_{\eta}(z)) z, \theta; \mu) + \mu W^u(\varepsilon - \mu \psi(\theta; \mu), \theta;
\mu).
\end{align*}
This implies
\begin{equation}\label{fa1-s2.3b}
{\bf X} = (1 + {\mathcal O}_{\theta, \mu}(\varepsilon) + z \hat
h(z, \theta; \mu))z -\mu {\mathcal O}_{\theta, \mu}(1)
\end{equation}
where $\hat h(z, \theta; \mu)$ is analytic in $z$, $\theta$, and $\mu$ and
satisfies $\hat h = {\mathcal O}_{z, \theta, \mu}(1)$. Now
substitute
\begin{equation*}
{\bf X} = \mu {\mathbb X}, \quad z = \mu Z
\end{equation*}
into~\pref{fa1-s2.3b} and note that $|Z| < \wh{K}_0 + 1$. We obtain the
claimed formula for ${\mathbb X}$.

For the ${\mathbb Y}$-component, we substitute $\xi$ and $\eta$ above
into~\pref{e:convertY} to obtain
\begin{equation*}
{\bf Y} = \varepsilon + ({\mathcal O}(\varepsilon)+z
h_{\eta}(z))
z  - \mu \psi(\theta; \mu) - \mu W^s((1 + {\mathcal O}(\varepsilon)+ z h_{\xi}(z)) z
-\mu \phi(\theta; \mu), \theta; \mu).
\end{equation*}
Set ${\bf Y} = \mu {\mathbb Y}$ and $z = \mu Z$ and note that $|Z| <
\wh{K}_0 + 1$. We obtain the claimed formula for ${\mathbb Y}$.
\end{proof}

\begin{corollary}\label{coro1-s2.3b}
On $\Sigma^-$, we have
\begin{equation*}
Z = (1 + {\mathcal O}_{\theta, p}(\varepsilon) + \mu {\mathcal
O}_{{\mathbb X}, \theta, p}(1))({\mathbb X} + {\mathcal
O}_{\theta, p}(1)).
\end{equation*}
\end{corollary}

\begin{proof}
We start with ~(\ref{fa1-s2.3b}). This
equality is invertible and we have
\begin{equation}\label{f3-s2.3b}
z = (1 + {\mathcal O}_{\theta, \mu}(\varepsilon) + {\bf W} \tilde
h({\bf W}, \theta; \mu)) {\bf W}
\end{equation}
where
\begin{equation*}
{\bf W} = {\bf X} + \mu {\mathcal O}_{\theta, \mu}(1)
\end{equation*}
and $\tilde h({\bf W}, \theta; \mu)$ is analytic in ${\bf W}$,
$\theta$, and $\mu$ and satisfies $\tilde h = {\mathcal O}_{{\bf W},
\theta, \mu}(1)$. Writing ~(\ref{f3-s2.3b}) in terms of $Z$ and
${\mathbb X}$, we have
\begin{equation*}
Z = (1 + {\mathcal O}_{\theta, p}(\varepsilon) + \mu {\mathcal
O}_{{\mathbb X}, \theta, p}(1))({\mathbb X} + {\mathcal
O}_{\theta, p}(1)).
\end{equation*}
\end{proof}

\begin{corollary}\label{coro2-s2.3b}
On $\Sigma^-$, we have
\begin{equation*}
{\mathbb Y} = \mu^{-1} \varepsilon + {\mathcal O}_{{\mathbb X},
\theta, p}(1).
\end{equation*}
\end{corollary}

\begin{proof}
We first regard ${\mathbb Y}$ as a
function of $Z$, $\theta$, and $p$ using the formula for ${\mathbb Y}$ in
Proposition ~\ref{prop-s2.3b} and then regard $Z$ as a function of
${\mathbb X}$, $\theta$, and $p$ using Corollary ~\ref{coro1-s2.3b}.
\end{proof}

\begin{remark}
Terms of the form $\mu {\mathcal O}_{{\mathbb X}, \theta, p}(1)$
are not equivalent to terms of the form ${\mathcal O}_{{\mathbb
X}, \theta, p}(\mu)$. A term of the form $\mu {\mathcal
O}_{{\mathbb X}, \theta, p}(1)$ has $C^3$ norm bounded above by $K
\mu$ while a term of the form ${\mathcal O}_{{\mathbb X}, \theta,
p}(\mu)$ has $C^3$ norm bounded above by $K(\varepsilon) \mu$. In
estimates in Section~\ref{s2.3}\ref{sss:conv1}
and~\ref{s2.3}\ref{sss:conv2}, we always have the former, not the
latter.
\end{remark}

\vspace{0.2cm}

\refstepcounter{psect} \label{sss:conv2} \noindent {\bfseries
\sffamily \ref{sss:conv2}. Conversion on} {\mathversion{bold} $\mathsf{\Sigma^+}$}. \ On
$\Sigma^+$ we need to write ${\mathbb X}$ and ${\mathbb Y}$ in
terms of $Z$.

\begin{proposition}\label{prop-s2.3c}
On $\Sigma^+$ we have
\begin{align*}
{\mathbb X} &= \mu^{-1} \varepsilon + {\mathcal O}_{Z, \theta, p}(1)\\
{\mathbb Y} &= (1 + {\mathcal O}_{\theta, p}(\varepsilon) + \mu
{\mathcal O}_{Z, \theta, p}(1)) Z - {\mathcal O}_{\theta, p}(1).
\end{align*}
\end{proposition}

\begin{proof}
On $\Sigma^+$, $s =L^+$. We have
\begin{equation}\label{f1-s3.1b}
\begin{aligned}
a(L^+) &= \varepsilon + Q_1(\varepsilon, 0) = \varepsilon +
{\mathcal
O}(\varepsilon^2)\\
b(L^+) &= Q_2(\varepsilon, 0) = {\mathcal O}(\varepsilon^2),
\end{aligned}
\end{equation}
and
\begin{equation}\label{f2-s3.1b}
u(L^+) = -1 + {\mathcal O}(\varepsilon), \quad v(L^+) = {\mathcal
O}(\varepsilon).
\end{equation}
Let $(z, \theta) \in \Sigma^+$.  We compute the values of ${\bf
X}$ and ${\bf Y}$ for this point. Using ~(\ref{f1-s3.1}) and
~(\ref{f1-s2.1a}), we have
\begin{align*}
\xi &= a(L^+) + v(L^+) z+ q_1(a(L^+)+ v(L^+) z, \ b(L^+)-u(L^+) z)\\
&= \varepsilon + {\mathcal O}(\varepsilon) z + q_1(a(L^+)+ v(L^+)
z,
\ b(L^+)-u(L^+) z) - q_1(a(L^+), b(L^+) ) \\
&= \varepsilon  + ({\mathcal O}(\varepsilon)+z k_{\xi}(z)) z.
\end{align*}
Similarly, we have
\begin{align*}
\eta &= b(L^+) - u(L^+) z + q_2(a(L^+)+ v(L^+) z, \ b(L^+)-u(L^+) z) \\
&= -u(L^+) z + q_2(a(L^+)+ v(L^+) z,
\ b(L^+)-u(L^+) z) - q_2(a(L^+), b(L^+) ) \\
&= (1 + {\mathcal O}(\varepsilon)+ z k_{\eta}(z))z.
\end{align*}
We now write ${\bf X}$ and ${\bf Y}$ in terms of $z$
using~\pref{e:convertX} and~\pref{e:convertY}.
The rest of the proof is similar to that of
Proposition ~\ref{prop-s2.3b}.
\end{proof}

\begin{corollary}\label{coro1-s2.3c}
If $L^+$ is sufficiently large, then ${\mathbb Y} > 1$ on
$\Sigma^+$.
\end{corollary}

\begin{proof}
This follows directly from the definition
of $\Sigma^+$.
\end{proof}

\section{Explicit computation of $\mscr{M}$ and $\mscr{N}$}
\label{s3}

In this section we explicitly compute the flow-induced maps
$\mscr{M}: \Sigma^- \to \Sigma^+$ and $\mscr{N}: \Sigma^+
\to \Sigma^-$. The map $\mscr{M}: \Sigma^- \to \Sigma^+$ is
computed in Section ~\ref{s3.0}. In Section ~\ref{s2.2} we study the
time-t map of equation ~(\ref{f1-s2.1d}). The map $\mscr{N}:
\Sigma^+ \to \Sigma^-$ is computed in Section ~\ref{s3.3}.

\subsection{Computing {\mathversion{bold}$\mscr{M}: \Sigma^- \to
\Sigma^+$}}\label{s3.0} \ Recall that $s = -L^-$ on $\Sigma^-$.
Let $q_0 = (-L^-, Z_0, \theta_0) \in \Sigma^-$ and let $(s(t),
Z(t), \theta(t))$ be the solution of
system~\pref{e:gnfa}--\pref{e:gnfc} initiated at the point $(-L^-, Z_0,
\theta_0)$. Let $\tilde t$ be the time such that $s(\tilde t) =
L^+$. By definition, $\mscr{M}(q_0) = (L^+, Z(\tilde t),
\theta(\tilde t))$. In this subsection we derive a specific form
of $\mscr{M}$ using $({\mathbb X}, \theta)$-coordinates to
uniquely locate points on $\Sigma^-$ and $(Z, \theta)$-coordinates
to uniquely locate points on $\Sigma^+$. Define
\begin{equation*}
K_1(\varepsilon) = - \rho A_L e^{\int_{0}^{L^+} E(s) \, ds}
\end{equation*}
where
\begin{equation*}
A_L = \int_{-L^-}^{L^+}(u(s) + v(s)) h(a(s), b(s)) e^{-
\int_{0}^{s} E(\tau) \, d\tau} \, ds
\end{equation*}
is obtained by changing the integral bounds of the improper
integral $A$ in ~(\ref{f2a-s1.1}) to $-L^-$ and $L^+$. Also define
\begin{equation*}
P_L  = e^{\int_{-L^-}^{L^+} E(s) \, ds}.
\end{equation*}

\begin{lemma}\label{lemma-added}
\begin{equation*}
P_L \sim \varepsilon^{\frac{\alpha}{\beta} - \frac{\beta}{\alpha}}
\ll 1, \quad K_1(\varepsilon) \sim
\varepsilon^{-\frac{\beta}{\alpha}}
\gg 1.
\end{equation*}
\end{lemma}

\begin{proof} Both estimates follow directly from Lemma
\ref{lem2-s2.2b} and the fact that
\begin{equation*}
\varepsilon \sim e^{-\alpha L^+} \sim e^{- \beta L^-}.
\end{equation*}
\end{proof}

\begin{proposition}\label{prop-s3.0}
Let $({\mathbb X}_0, \theta_0) \in \Sigma^-$ and write $(\hat Z,
\hat \theta) = \mscr{M}({\mathbb X}_0, \theta_0)$. We have
\begin{equation}\label{f6-s3.1c}
\begin{aligned}
\hat \theta &= \theta_0 + \omega (L^+ + L^-)  + {\mathcal O}_{{\mathbb X}_0, \theta_0, p}(\mu) \\
\hat Z &=  K_1(\varepsilon)(1+c_1 \sin \theta_0 + c_2 \cos
\theta_0) + P_L ({\mathbb X}_0 + {\mathcal O}_{\theta_0, p}(1) +
{\mathcal O}_{{\mathbb X}_0, \theta_0, p}(\varepsilon) + {\mathcal
O}_{{\mathbb X}_0, \theta_0, p}(\mu))
\end{aligned}
\end{equation}
where $c_1$ and $c_2$ are constants satisfying
\begin{equation*}
\frac{1}{4}< \sqrt{c_1^2 + c_2^2} < \frac{1}{2}.
\end{equation*}
\end{proposition}
\begin{proof}
Using~\pref{e:gnfc},
we have
\begin{equation*}
\theta(t) = \theta_0 + \omega t.
\end{equation*}
Integrating~\pref{e:gnfa} and~\pref{e:gnfb}, for $t \in[-2 L^-, 2
L^+]$ we have
\begin{equation*}
s(t) = -L^- + t + {\mathcal O}_{t, Z_0, \theta_0, p}(\mu).
\end{equation*}
Inverting the last equality, we obtain
\begin{equation*}
t(s) = s + L^- + {\mathcal O}_{s, Z_0, \theta_0, p}(\mu).
\end{equation*}
Substituting $\theta(t)$ and $t(s)$ into~\pref{e:gnfb}, we obtain
\begin{equation}\label{sharp}
\frac{dZ}{ds} = E(s) Z - (u(s) +  v(s))(\rho {\mathbb H}(s, 0) +
\sin(\theta_0 +\omega L^- + \omega s)) + {\mathcal O}_{s, Z_0,
\theta_0, p}(\mu).
\end{equation}
Note that in~\pref{sharp}, $(s, Z_0, \theta_0, p)$ is such that $s
\in [-2L^-, 2L^+]$, $(Z_0, \theta_0) \in \Sigma^-$, and $p = \ln
\mu \in (-\infty, \ln \mu_0]$. Using~(\ref{sharp}), we obtain
\begin{equation}\label{fa1-s3.1c}
Z(s) =P_s \cdot (Z_0 - \Phi_s(\theta_0) + {\mathcal O}_{s, Z_0,
\theta_0, p}(\mu))
\end{equation}
where
\begin{equation}\label{fadd-s3.0}
\begin{gathered}
P_s = e^{\int_{-L^-}^{s} E(\tau) \, d\tau} \\
\Phi_s(\theta) = \int_{-L^-}^{s}(u(\tau) +  v(\tau)) (\rho
{\mathbb H}(\tau, 0)+ \sin(\theta + \omega L^- + \omega \tau))
\cdot e^{-\int_{-L^-}^\tau E(\hat \tau) \, d\hat \tau} \, d\tau.
\end{gathered}
\end{equation}

From ~(\ref{fa1-s3.1c}), it follows that
\begin{equation}\label{f3-s3.1c}
\begin{aligned}
\hat \theta &= \theta_0 + \omega (L^+ + L^-)  + {\mathcal O}_{Z_0, \theta_0, p}(\mu) \\
\hat Z &= P_L(Z_0 - \Phi_{L^+}(\theta_0) + {\mathcal O}_{Z_0,
\theta_0, p}(\mu)).
\end{aligned}
\end{equation}

We want to write the right-hand side of ~(\ref{f3-s3.1c}) in
$({\mathbb X}_0, \theta_0)$-coordinates. Using Corollary ~\ref{coro1-s2.3b}, we
have
\begin{equation}\label{f5-s3.1c}
\begin{aligned}
\hat \theta &= \theta_0 + \omega (L^+ + L^-)  + {\mathcal O}_{{\mathbb X}_0, \theta_0, p}(\mu) \\
\hat Z &=  P_L \left({\mathbb X}_0  - \Phi_{L^+}(\theta_0) +
{\mathcal O}_{\theta_0, p}(1) + {\mathcal O}_{{\mathbb X}_0,
\theta_0, p}(\varepsilon) + {\mathcal O}_{{\mathbb X}_0, \theta_0,
p}(\mu)\right).
\end{aligned}
\end{equation}
Let $K_2$ be such that
\begin{equation*}
|{\mathbb X}_0 + {\mathcal O}_{\theta_0, p}(1) + {\mathcal
O}_{{\mathbb X}_0, \theta_0, p}(\varepsilon) + {\mathcal
O}_{{\mathbb X}_0, \theta_0, p}(\mu)| < K_2
\end{equation*}
on $\Sigma^-$ and observe that by letting
\begin{equation}\label{f4-s3.1c}
K_0(\varepsilon) = \max_{\substack{\theta \in \mbb{S}^{1}\\ s \in
[-2L^-, 2L^+]}} 2 \left| P_s (K_2 - \Phi_s(\theta)) \right|,
\end{equation}
we conclude from ~(\ref{fa1-s3.1c}) that all solutions of
system~\pref{e:gnfa}--\pref{e:gnfc} initiated inside of $\Sigma^-$
will stay inside of
\begin{equation*}
\{(s, Z, \theta) : s \in [-2L^-, 2L^+], \; |Z| < K_0(\varepsilon)
\}
\end{equation*}
before reaching $s = L^+$. To finish the proof of
Proposition~\ref{prop-s3.0}, it now suffices for us to prove the
following lemma.

\begin{lemma}\label{prop1-s3.2}
For $\rho \in [\rho_1, \rho_2]$, we have
\begin{equation*}
-P_L \Phi_{L^+}(\theta) = K_1(\varepsilon) (1 + c_1 \sin \theta +
c_2 \cos \theta)
\end{equation*}
where $c_1$ and $c_2$ are constants satisfying
\begin{equation*}
\frac{1}{4} < \sqrt{c_1^2 + c_2^2} < \frac{1}{2}.
\end{equation*}
\end{lemma}
\begin{proof}[Proof of Lemma~\ref{prop1-s3.2}]
Recall that in
~(\ref{fadd-s3.0}), ${\mathbb H}(s, 0) = h(a(s), b(s))$. We have
\begin{align*}
P_L \Phi_{L^+}(\theta) &= e^{\int_{0}^{L^+} E(\tau) \, d\tau}
\cdot \int_{-L^-}^{L^+} (u(s) + v(s)) (\rho h(a(s), b(s)) \\
&\hspace{4cm} {}+ \sin
(\theta + \omega L^- + \omega s)) e^{- \int_{0}^{s}
E(\tau) \, d\tau} \, ds \\
&= e^{\int_{0}^{L^+} E(\tau) \, d\tau} \cdot (\rho A_L + (C_L \cos
\omega L^- - S_L \sin \omega L^- ) \sin \theta \\
&\hspace{3.2cm} {}+ (S_L \cos \omega
L^- + C_L \sin \omega L^-) \cos \theta)
\end{align*}
where
\begin{align*}
A_L &= \int_{-L^-}^{L^+} (u(s) + v(s))h(a(s), b(s)) e^{-
\int_{0}^{s}
E(\tau) \, d\tau} \, ds  \\
C_L &= \int_{-L^-}^{L^+} (u(s) + v(s)) \cos (\omega s) e^{- \int_{0}^{s} E(\tau) \, d\tau} \, ds \\
S_L &= \int_{-L^-}^{L^+} (u(s) + v(s)) \sin (\omega s) e^{-
\int_{0}^{s} E(\tau) \, d\tau} \, ds.
\end{align*}
Observe that $A$, $C$, and $S$ in~\bpref{li:h2} are obtained by
letting $L^{\pm} = \infty$ in $A_L$, $C_L$, and $S_L$. We now
write
\begin{equation}\label{f1-prop3.1}
P_L \Phi_{L^+}(\theta) = \rho A_L e^{\int_{0}^{L^+} E(\tau) \,
d\tau} \cdot (1+  c_1 \sin \theta + c_2 \cos \theta)
\end{equation}
where
\begin{align*}
c_1 &= \frac{(C_L \cos \omega L^- - S_L \sin \omega L^-)}{A_L \rho} \\
c_2 &= \frac{(S_L \cos \omega L^- + C_L \sin \omega L^-)}{A_L
\rho}.
\end{align*}
We have
\begin{equation*}
c_1^2 + c_2^2 = \frac{(C_L^2 + S_L^2)}{A_L^2 \rho^2}.
\end{equation*}
Using~\bpref{li:h2}, for $L^{\pm}$ sufficiently large we have
\begin{gather*}
|A_L -A| < \frac{1}{100} |A| \\
\left| \sqrt{C_L^2 + S_L^2} -
\sqrt{C^2 + S^2} \right| < \frac{1}{100} \sqrt{C^2 + S^2}.
\end{gather*}
Therefore, for $\rho \in [\rho_1, \rho_2]$, where
\begin{equation}\label{f2-prop3.1}
\rho_1 = - \frac{202}{99} \frac{\sqrt{C^2 + S^2}}{A}, \quad \rho_2
= - \frac{396}{101} \frac{\sqrt{C^2 + S^2}}{A},
\end{equation}
we have
\begin{equation*}
\frac{1}{4} < \sqrt{c_1^2 + c_2^2} < \frac{1}{2}.
\end{equation*}
Notice that because of the way in which $\rho_1$ and $\rho_2$ are defined,
we have $- \rho A_L >
0$. We also have
\begin{equation*}
K_1(\varepsilon) = - \rho A_L e^{\int_{0}^{L^+} E(\tau) \, d\tau}
\sim e^{-\frac{\beta}{\alpha}}
\end{equation*}
from Lemma \ref{lemma-added}. Equation ~(\ref{f1-prop3.1}) for
$P_L \Phi_{L^+}(\theta)$ is now in the asserted form.
\end{proof}

By using Lemma ~\ref{prop1-s3.2}, we can now rewrite
~(\ref{f5-s3.1c}) as ~(\ref{f6-s3.1c}). This finishes the proof of
Proposition ~\ref{prop-s3.0}.
\end{proof}

\begin{remark}
Observe that in formula ~(\ref{f6-s3.1c})
for $\hat Z$, the term with $K_1(\varepsilon)$ in front
dominates the second term because $K_1(\varepsilon) \gg P_L$. The
inclusion $\mscr{M}(\Sigma^-) \subset \Sigma^+$ follows
directly from ~(\ref{f6-s3.1c}).
\end{remark}

\subsection{On the time-t map of equation
~(\ref{f1-s2.1d})}\label{s2.2} The computation of $\mscr{N}:
\Sigma^+ \to \Sigma^-$ contains two major steps. The first step is
to compute the time-t map of equation ~(\ref{f1-s2.1d}) inside
$\mscr{U}_{\varepsilon}$. This is done in Section ~\ref{s2.2}. The
second step is to compute the time it takes for a solution of
equation ~(\ref{f1-s2.1d}) initiated in $\Sigma^+$ to reach
$\Sigma^-$. This is done in Section ~\ref{s3.3}. These
computations are technically involved because we need to control
the $C^3$ norms of the map $\mscr{N}$ on $\Sigma^+ \times
(-\infty, \ln \mu_0]$, where the interval in the product is the
domain of the parameter $p$.

We start with the first step. Let $W(\Sigma^+)$ be a small open
neighborhood surrounding $\Sigma^+$ in the space $({\mathbb X},
{\mathbb Y}, \theta)$. In this subsection we let $({\mathbb X}_0,
{\mathbb Y}_0, \theta_0) \in W(\Sigma^+)$ and regard $p = \ln \mu
\in (-\infty, \ln \mu_0]$ as the parameter of equation
~(\ref{f1-s2.1d}). We study the time-t map of equation
~(\ref{f1-s2.1d}) assuming that up to time $t$, all solutions
initiated from $W(\Sigma^+)$ are completely contained inside
$\mscr{U}_{\varepsilon}$. Recall that in equation
~(\ref{f1-s2.1d}), ${\mathbb F}({\mathbb X}, {\mathbb Y}, \theta; \mu)$ and
${\mathbb G}({\mathbb X}, {\mathbb Y}, \theta; \mu)$
are analytic on
\begin{equation*}
{\mathbb D} = \{ ({\mathbb X}, {\mathbb Y}, \theta, \mu): \mu
\in [0, \mu_0], \; ({\mathbb X}, {\mathbb Y}, \theta) \in
\mscr{U}_{\varepsilon} \}
\end{equation*}
where
\begin{equation*}
\mscr{U}_{\varepsilon} = \{ ({\mathbb X}, {\mathbb Y},
\theta) : \| ({\mathbb X}, {\mathbb Y}) \| < 2 \varepsilon
\mu^{-1}, \; \theta \in \mbb{S}^{1} \}.
\end{equation*}
For $q_0 = ({\mathbb X}_0, {\mathbb Y}_0, \theta_0) \in
W(\Sigma^+)$, let
\begin{equation*}
q(t, q_0; \mu) = ({\mathbb X}(t, q_0; \mu), {\mathbb Y}(t, q_0;
\mu), \theta(t, q_0; \mu))
\end{equation*}
be the solution of equation ~(\ref{f1-s2.1d}) initiated from $q_0$
at $t = 0$. Using ~(\ref{f1-s2.1d}), we have
\begin{equation}\label{f2-s2.2}
\begin{aligned}
{\mathbb X}(t, q_0; \mu) &= {\mathbb X}_0
e^{\int_0^t (- \alpha + \mu {\mathbb F}(q(s, q_0; \mu); \mu)) \, ds} \\
{\mathbb Y}(t, q_0; \mu) &= {\mathbb Y}_0 e^{\int_0^t (\beta +
\mu {\mathbb G}(q(s, q_0; \mu); \mu)) \, ds } \\
\theta(t, q_0; \mu) &= \theta_0 + \omega t.
\end{aligned}
\end{equation}
We now introduce the functions $U(t, q_0; \mu)$ and $V(t, q_0; \mu)$
and rewrite ~(\ref{f2-s2.2}) as
\begin{equation}\label{f3-s2.2}
\begin{aligned}
{\mathbb X}(t, q_0; \mu) &= {\mathbb X}_0
e^{(- \alpha + U(t, q_0; \mu)) t } \\
{\mathbb Y}(t, q_0; \mu) &= {\mathbb Y}_0 e^{(\beta + V(t, q_0; \mu))t } \\
\theta(t, q_0; \mu) &= \theta_0 + \omega t.
\end{aligned}
\end{equation}
Using~(\ref{f3-s2.2}), we have
\begin{equation}\label{f4-s2.2}
\begin{aligned}
U(t, q_0; \mu) &= t^{-1} \ln \frac{{\mathbb X}(t, q_0;
\mu)}{{\mathbb X}_0} + \alpha \\
V(t, q_0; \mu) &= t^{-1} \ln \frac{{\mathbb Y}(t, q_0;
\mu)}{{\mathbb Y}_0} - \beta. \\
\end{aligned}
\end{equation}
We also have
\begin{equation}\label{f5-s2.2}
\begin{aligned}
U(t, q_0; \mu) &= t^{-1} \int_0^t \mu {\mathbb F}(q(s, q_0; \mu);
\mu) \, ds \\
V(t, q_0; \mu) &= t^{-1} \int_0^t \mu {\mathbb G}(q(s, q_0; \mu);
\mu) \, ds.
\end{aligned}
\end{equation}
In the next proposition we regard $U = U(t, q_0; \mu)$ and $V = V(t,
q_0; \mu)$ as functions of $t$, $q_0$, and $p$ and we write
$U = U_{t, q_0, p}$ and $V = V_{t, q_0, p}$, respectively. We define the domain
of these two functions as follows.  Let
\begin{equation*}
{\mathbb D}_{t, q_0, p} = \{ q_0 \in W(\Sigma^+), \; p \in
(-\infty, \ln \mu_0], \; t \in [1, T(q_0, p)] \}
\end{equation*}
where the upper bound $T(q_0, p)$ on $t$ is designed to keep the
solution inside $\mscr{U}_{\varepsilon}$.
\begin{proposition}\label{prop-2.4}
There exists $K > 0$ such that
\begin{equation*}
\|U_{t, q_0, p}\|_{C^3({\mathbb D}_{t, q_0, p})}< K \mu, \quad
\|V_{t, q_0, p}\|_{C^3({\mathbb D}_{t, q_0, p})} < K \mu.
\end{equation*}
\end{proposition}

Proposition \ref{prop-2.4} is proved in Section \ref{sadd-1}.

\begin{remark}
By combining Proposition ~\ref{prop-2.4}
and ~(\ref{f3-s2.2}), we can now write the time-t map from
$W (\Sigma^+)$ to $\mscr{U}_{\varepsilon}$ as
\begin{equation}\label{f10-s2.2}
\begin{aligned}
{\mathbb X}(t, {\mathbb X}_0, {\mathbb Y}_0, \theta_0; \mu) &=
{\mathbb X}_0
e^{(- \alpha + {\mathcal O}_{t, {\mathbb X}_0, {\mathbb Y}_0, \theta_0, p}(\mu)) t } \\
{\mathbb Y}(t, {\mathbb X}_0, {\mathbb Y}_0, \theta_0; \mu) &=
{\mathbb Y}_0 e^{(\beta + {\mathcal O}_{t, {\mathbb X}_0, {\mathbb Y}_0, \theta_0, p}(\mu))t } \\
\theta(t, {\mathbb X}_0, {\mathbb Y}_0, \theta_0; \mu) &=
\theta_0 + \omega t.
\end{aligned}
\end{equation}
\end{remark}

\subsection{Estimates on $T(Z_0, \theta_0, p)$}\label{s3.3}
For $q_0 = (Z_0, \theta_0) \in \Sigma^+$, let $q(t, q_0; \mu)$ be
the solution of equation ~(\ref{f1-s2.1d}) initiated at $q_0$ and
let $T$ be the time this solution reaches $\Sigma^-$. In this
subsection we regard $T$ as a function of $Z_0$, $\theta_0$, and $p$ and
we obtain a well-controlled formula for $T$ that is explicit in
the variables $Z_0$, $\theta_0$, and $p$. Since the images of
$\mscr{M}$ are expressed in $(Z, \thet)$-coordinates through
~(\ref{f6-s3.1c}), we must write the initial conditions for
$\mscr{N}$ in $(Z, \thet)$-coordinates on $\Sigma^+$ to
facilitate the intended composition of $\mscr{N}$ and
$\mscr{M}$.

Estimates on $T(Z_0, \theta_0, p)$ are complicated partly because
as a function of $Z_0$ and $\theta_0$, it is implicitly defined
through equations written in $(\mbb{X}, \mbb{Y},
\thet)$-coordinates on $\Sigma^{\pm}$. The computational process
therefore must involve ~(\ref{f10-s2.2}) and the coordinate
transformations on $\Sigma^{\pm}$ presented in
Sections~\ref{s2.3}B and~\ref{s2.3}C.  Before presenting the
desired quantitative estimates, we explain how to obtain $T(Z_0,
\theta_0, p)$ in a conceptual way. Using ~(\ref{f3-s2.2}), we
obtain
\begin{equation}\label{f1-T}
\begin{aligned}
{\mathbb X}(T, {\mathbb X}_0, {\mathbb Y}_0, \theta_0; \mu) &=
{\mathbb X}_0
e^{(- \alpha + U(T, {\mathbb X}_0, {\mathbb Y}_0, \theta_0; p) T } \\
{\mathbb Y}(T, {\mathbb X}_0, {\mathbb Y}_0, \theta_0; \mu) &=
{\mathbb Y}_0 e^{(\beta + V(T, {\mathbb X}_0, {\mathbb Y}_0, \theta_0; p))T } \\
\theta(T, {\mathbb X}_0, {\mathbb Y}_0, \theta_0; \mu) &=
\theta_0 + \omega T.
\end{aligned}
\end{equation}
In~(\ref{f1-T}), ${\mathbb X}_0$ and ${\mathbb Y}_0$ are {\it not} independent variables.  These quantities satisfy
\begin{equation}\label{f1-s3.3}
\begin{aligned}
{\mathbb X}_0 &= \mu^{-1} \varepsilon + {\mathcal O}_{Z_0, \theta_0, p}(1)\\
{\mathbb Y}_0 &= (1 + {\mathcal O}_{\theta_0, p}(\varepsilon) +
\mu {\mathcal O}_{Z_0, \theta_0, p}(1)) Z_0 - {\mathcal
O}_{\theta_0, p}(1)
\end{aligned}
\end{equation}
by Proposition ~\ref{prop-s2.3c}.  We write
\begin{align*}
{\mathbb X}(T) &= {\mathbb X}(T, {\mathbb X}_0, {\mathbb Y}_0,
\theta_0; \mu) \\
{\mathbb Y}(T) &= {\mathbb Y}(T, {\mathbb
X}_0, {\mathbb Y}_0, \theta_0; \mu) \\
\theta(T) &= \theta_0 +
\omega T.
\end{align*}
By definition, ${\mathbb X}(T)$, ${\mathbb Y}(T)$, and $\theta(T)$ are
also related through Corollary ~\ref{coro2-s2.3b}. For the benefit
of a clear exposition, we write the conclusion of Corollary
~\ref{coro2-s2.3b} as
\begin{equation*}
{\mathbb Y} = \varepsilon \mu^{-1} + {\bf f}({\mathbb X}, \theta;
p)
\end{equation*}
where
\begin{equation*}
{\bf f}({\mathbb X}, \theta; p) = {\mathcal O}_{{\mathbb X},
\theta, p}(1).
\end{equation*}
We have
\begin{equation}\label{f2-T}
{\mathbb Y}(T) = \varepsilon \mu^{-1} + {\bf f}({\mathbb X}(T),
\theta(T); p).
\end{equation}

We use~(\ref{f1-T}) to implicitly define
$T(Z_0, \theta_0; p)$. We have
\begin{equation}\label{f3-T}
{\mathbb Y}(T) = {\mathbb Y}_0 e^{(\beta + V(T, {\mathbb X}_0,
{\mathbb Y}_0, \theta_0; p))T}.
\end{equation}

The right-hand side of ~(\ref{f3-T}) is relatively simple: we only
need to substitute for ${\mathbb X}_0$ and ${\mathbb Y}_0$ using
~(\ref{f1-s3.3}). The left-hand side of ~(\ref{f3-T}) is
conceptually more complicated. We need to
\begin{enumerate}
\item
Write ${\mathbb Y}(T)$ as a function of ${\mathbb X}(T)$,
$\theta(T)$, and $p$ using ~(\ref{f2-T}).
\item
Substitute for ${\mathbb X}(T)$ and $\theta(T)$ using~(\ref{f1-T}), thereby obtaining ${\mathbb
Y}(T)$ in terms of $T$, ${\mathbb X}_0$, ${\mathbb Y}_0$, $\theta_0$,
and $p$.
\item
Use~\pref{f1-s3.3} to write ${\mathbb X}_0$ and
${\mathbb Y}_0$ in terms of $Z_0$ and $\theta_0$.
\end{enumerate}
After all of these substitutions are made, we regard ~(\ref{f3-T}) as the
equation that implicitly defines $T(Z_0, \theta_0; p)$. We use
this equation as the basis for the computation of $T(Z_0,
\theta_0; p)$.

\begin{proposition} \label{prop1-s3.3}
As a function of $Z_0$, $\theta_0$, and $p$, the map $T$ satisfies
\begin{equation*}
\|T - \frac{1}{\beta} \ln \mu^{-1}\|_{C^3} < K.
\end{equation*}
\end{proposition}

Proposition \ref{prop1-s3.3} is proved in Section \ref{sadd-2}.

\subsection{Computing {\mathversion{bold} $\mscr{N} : \Sigma^+ \to \Sigma^-$}}

We derive a formula for the induced map $\mscr{N}_p: \Sigma^+ \to
\Sigma^-$. For $(Z_0, \theta_0) \in \Sigma^+$, we write $({\mathbb
X}_1, \theta_1) = \mscr{N}_p(Z_0, \theta_0)$. We start with $U$
and $V$ in ~(\ref{f1-T}).
\begin{lemma}\label{lem5-T}
On $\Sigma^+ \times (-\infty, \ln \mu_0]$, we have
\begin{align*}
U(T, {\mathbb X}_0, {\mathbb Y}_0, \theta_0; p) &= \mu {\mathcal
O}_{Z_0, \theta_0, p}(1) \\
V(T, {\mathbb X}_0, {\mathbb Y}_0, \theta_0; p) &= \mu {\mathcal
O}_{Z_0, \theta_0, p}(1).
\end{align*}
\end{lemma}

\begin{proof}
We write $U$ and $V$ as functions of $(Z_0, \theta_0, p)$ using
Proposition ~\ref{prop1-s3.3} for $T(Z_0, \theta_0; p)$ and
~(\ref{f1-s3.3}) for ${\mathbb X}_0$ and ${\mathbb Y}_0$. This
lemma is established by applying the chain rule and using
Proposition ~\ref{prop-2.4}, Proposition ~\ref{prop1-s3.3}, and
~(\ref{f1-s3.3}).
\end{proof}

\begin{proposition}\label{prop2-s3.3}
The flow-induced map $\mscr{N}_p: \Sigma^+ \to \Sigma^-$ is given
by
\begin{equation}\label{f8-s2.2}
\begin{aligned}
{\mathbb X}_1 &=  \left(\frac{\mu}{\varepsilon + \mu{\mathcal
O}_{Z_0, \theta_0, p}(1)}\right)^{\frac{\tilde \alpha}{\tilde
\beta}-1}\left([1 + {\mathcal O}_{\theta_0, p}(\varepsilon) + \mu
{\mathcal O}_{Z_0, \theta_0, p}(1)] Z_0 - {\mathcal O}_{\theta_0,
p}(1) \right)^{\frac{\tilde \alpha}{\tilde \beta}}\\
\theta_1 &= \theta_0 + \frac{\omega} {\beta + \mu {\mathcal
O}_{Z_0, \theta_0, p}(1)} \ln \frac{(\varepsilon + \mu {\mathcal
O}_{Z_0, \theta_0, p}(1))\mu^{-1}}{[1 + {\mathcal O}_{\theta_0,
p}(\varepsilon) + \mu {\mathcal O}_{Z_0, \theta_0, p}(1)] Z_0 -
{\mathcal O}_{\theta_0, p}(1)}
\end{aligned}
\end{equation}
where
\begin{equation*}
\tilde \alpha = \alpha + \mu {\mathcal O}_{Z_0, \theta_0, p}(1),
\quad \tilde \beta = \beta + \mu{\mathcal O}_{Z_0, \theta_0,
p}(1).
\end{equation*}
\end{proposition}

\begin{proof}
Using ~(\ref{f2-T}), ~(\ref{f3-T}) and Lemma ~\ref{lem5-T}, we
have
\begin{equation}\label{f6-T}
\begin{aligned}
T &= \frac{1} {\beta + \mu {\mathcal O}_{Z_0, \theta_0, p}(1)}
\ln \frac{Y(T)}{Y_0} \\
&= \frac{1} {\beta + \mu {\mathcal O}_{Z_0, \theta_0, p}(1)} \ln
\frac{(\varepsilon + \mu {\bf f}({\mathbb X}(T), \theta(T);
p))\mu^{-1} }{{\mathbb Y}_0}.
\end{aligned}
\end{equation}
By using Proposition ~\ref{prop1-s3.3} and the fact that ${\bf
f}({\mathbb X}, \theta; p) = {\mathcal O}_{{\mathbb X}, \theta,
p}(1)$, we have
\begin{equation*}
{\bf f}({\mathbb X}(T), \theta(T); p) = {\mathcal O}_{Z_0,
\theta_0, p}(1).
\end{equation*}
Now ~(\ref{f6-T}) gives
\begin{equation}\label{f7-T}
T = \frac{1} {\beta + \mu {\mathcal O}_{Z_0, \theta_0, p}(1)} \ln
\frac{\mu^{-1} (\varepsilon + \mu {\mathcal O}_{Z_0, \theta_0,
p}(1)) }{[1 + {\mathcal O}_{\theta_0, p}(\varepsilon) + \mu
{\mathcal O}_{Z_0, \theta_0, p}(1)] Z_0 - {\mathcal O}_{\theta_0,
p}(1)}.
\end{equation}
Here we use ~(\ref{f1-s3.3}) for ${\mathbb Y}_0$.

The desired formula for $\theta_1$ now follows from $\theta_1 =
\theta_0 + \omega T$. For ${\mathbb X}_1$ we use
\begin{equation*}
{\mathbb X}_1 = \mu^{-1} (\varepsilon + \mu {\mathcal O}_{Z_0,
\theta_0, p}(1)) e^{-(\alpha + \mu {\mathcal O}_{Z_0, \theta_0,
p}(1))T}
\end{equation*}
and substitute for $T$ using ~(\ref{f7-T}).
\end{proof}

\section{Proof of Theorem ~\ref{th-s1.1}} \label{s4}

In Subsection~\ref{s4.1} we compute $\mscr{F}_{p} = \mscr{N} \circ
\mscr{M}$ by using Propositions~\ref{prop2-s3.3}
and~\ref{prop-s3.0}. In Subsection~\ref{ss:2pf} we apply the
theory of rank one maps to the family $\{ \mscr{F}_{p} \}$,
thereby proving the existence of rank one chaos as claimed in
Theorem ~\ref{th-s1.1}.

\subsection{The flow-induced map {\mathversion{bold} $\mscr{F} = \mscr{N} \circ \mscr{M}$}}
\label{s4.1}

We regard $p$ as the fundamental parameter of the flow-induced map $\mscr{F} : \Si^{-} \to
\Si^{-}$.  For $(\mbb{X}_{0}, \thet_{0}) \in \Si^{-}$, let $(\mbb{X}_{1}, \thet_{1}) =
(\mscr{N} \circ \mscr{M}) (\mbb{X}_{0}, \thet_{0})$.  We compute $\mscr{F}_{p} :
(\mbb{X}_{0}, \thet_{0}) \mapsto (\mbb{X}_{1}, \thet_{1})$ by combining~\pref{f8-s2.2}
and~\pref{f6-s3.1c}.

\begin{proposition}\label{prop-s4.1}
The map $\mscr{F}_p: \Sigma^- \to \Sigma^-$ is given by
\begin{align}
\sublabon{equation}
{\mathbb X}_1 &=  (\mu (\varepsilon + {\mathcal O}_{{\mathbb
X}_0,\theta_0, p}(\mu) )^{-1})^{\tfrac{\tilde \alpha}{\tilde
\beta}-1}
\label{e:fx} \\
&\qquad \times \left((1 + {\mathcal O}^{\partial \mbb{X}_{0} \langle \mu \rangle}_{{\mathbb X}_0,
\theta_0,
p}(\varepsilon) + {\mathcal O}_{{\mathbb X}_0, \theta_0, p}(\mu))
\mscr{Z} - {\mathcal O}^{\partial \mbb{X}_{0} \langle \mu \rangle}_{{\mathbb X}_0, \theta_0, p}(1)
\right)^{\tfrac{\tilde \alpha}{\tilde \beta}}
\notag \\
\theta_1 &= \theta_0 + \omega(L^+ + L^-)  + {\mathcal O}_{{\mathbb
X}_0, \theta_0, p}(\mu)
\label{e:fthet} \\
&\qquad {}+ \frac{\omega}{\beta + {\mathcal O}_{{\mathbb X}_0,
\theta_0, p}(\mu)} \ln \frac{(\varepsilon + {\mathcal O}_{{\mathbb
X}_0, \theta_0, p}(\mu))\mu^{-1}}{(1 + {\mathcal O}^{\partial \mbb{X}_{0} \langle \mu \rangle}_{{\mathbb
X}_0, \theta_0, p}(\varepsilon) + {\mathcal O}_{{\mathbb X}_0,
\theta_0, p}(\mu)) \mscr{Z} - {\mathcal O}^{\partial \mbb{X}_{0} \langle \mu \rangle}_{{\mathbb X}_0,
\theta_0, p}(1)}
\notag
\end{align}
\sublaboff{equation}
where
\begin{gather*}
\mscr{Z} = K_1(\varepsilon)(1+c_1 \sin \theta_0 + c_2
\cos \theta_0) + P_L[{\mathbb X}_0 + {\mathcal O}_{\theta_0, p}(1) +
{\mathcal O}_{{\mathbb X}_0, \theta_0,
p}(\varepsilon) + {\mathcal O}_{{\mathbb X}_0, \theta_0, p}(\mu)]
\\
\begin{aligned}
\tilde \alpha &= \alpha + {\mathcal O}_{{\mathbb X}_0, \theta_0,
p}(\mu) \\
\tilde \beta &= \beta + {\mathcal O}_{{\mathbb X}_0, \theta_0,
p}(\mu)
\end{aligned}
\end{gather*}
and the superscript $\partial \mbb{X}_{0} \langle \mu \rangle$ on a given term indicates
that the partial derivative of the term with respect to $\mbb{X}_{0}$ is $\mcal{O} (\mu)$.
We also have
\begin{equation*}
K_1(\varepsilon) \sim \varepsilon^{- \frac{\beta}{\alpha}}, \quad
\frac{1}{4} < \sqrt{c_1^2 + c_2^2} < \frac{1}{2}.
\end{equation*}
\end{proposition}

\begin{proof}
We first examine the formulas for $\tilde
\alpha$ and $\tilde \beta$. The error terms in Proposition
~\ref{prop2-s3.3} have the form
\begin{equation*}
\mu {\mathcal O}_{\hat Z, \hat \theta, p}(1)
\end{equation*}
and $\hat Z$ and $\hat \theta$ are given in terms of ${\mathbb X}_0$,
$\theta_0$, and $p$ by ~(\ref{f6-s3.1c}). Using  ~(\ref{f6-s3.1c}), we see
that the $C^3$ norms of $\hat Z$ and $\hat \theta$ are $<
K(\varepsilon)$. It follows from the chain rule that
\begin{equation*}
\mu {\mathcal O}_{\hat Z, \hat \theta, p}(1) = {\mathcal
O}_{{\mathbb X}_0, \theta_0, p}(\mu).
\end{equation*}

We follow the same line of reasoning to compute ${\mathbb X}_1$ and
$\theta_1$.  We replace $Z_{0}$ and $\thet_{0}$ with $\hat{Z}$ and $\hat{\thet}$
in~\pref{f8-s2.2} and then substitute for $\hat{Z}$ and
$\hat{\thet}$ using~\pref{f6-s3.1c}.  Using~\pref{f8-s2.2}, we have
\begin{equation}\label{e:fprelim}
\begin{aligned}
{\mathbb X}_1 &=  \left(\frac{\mu}{\varepsilon + \mu{\mathcal
O}_{\hat Z, \hat \theta, p}(1)}\right)^{\frac{\tilde
\alpha}{\tilde \beta}-1}\left([1 + {\mathcal O}_{\hat \theta,
p}(\varepsilon) + \mu {\mathcal O}_{\hat Z, \hat \theta, p}(1)]
\hat Z - {\mathcal O}_{\hat \theta,
p}(1) \right)^{\frac{\tilde \alpha}{\tilde \beta}} \\
\theta_1 &= \hat \theta + \frac{\omega} {\beta + \mu {\mathcal
O}_{\hat Z, \hat \theta, p}(1)} \ln \frac{(\varepsilon + \mu
{\mathcal O}_{\hat Z, \hat \theta, p}(1))\mu^{-1}}{[1 + {\mathcal
O}_{\hat \theta, p}(\varepsilon) + \mu {\mathcal O}_{\hat Z, \hat
\theta, p}(1)] \hat Z - {\mathcal O}_{\hat \theta, p}(1)}.
\end{aligned}
\end{equation}

In~\pref{e:fprelim}, terms of the form $ \mu {\mathcal O}_{\hat Z, \hat \theta, p}(1)$ are
rewritten in the form ${\mathcal O}_{{\mathbb X}_0, \theta_0, p}(\mu)$ using
~(\ref{f6-s3.1c}). Terms of the form ${\mathcal O}_{\hat \theta, p}(\varepsilon)$ are
rewritten in the form ${\mathcal O}^{\partial \mbb{X}_{0} \langle \mu \rangle}_{{\mathbb
    X}_0, \theta_0, p}(\varepsilon)$ because the $C^3$ norm of $\hat \theta$ is bounded by
a constant $K$ independent of $\varepsilon$ and because $\frac{\partial
  \hat{\thet}}{\partial \mbb{X}_{0}} = \mcal{O} (\mu)$.  Reasoning analogously, terms of
the form ${\mathcal O}_{\hat \theta, p}(1)$ are rewritten in the form ${\mathcal
  O}^{\partial \mbb{X}_{0} \langle \mu \rangle}_{{\mathbb X}_0, \theta_0, p}(1)$.
\end{proof}

\subsection{Proof of Theorem ~\ref{th-s1.1}}
We are finally ready to prove Theorem ~\ref{th-s1.1}.

\smallskip

\noindent {\bf The two-parameter family $\{ \mscr{F}_{a, b_{n}}
\}$.} \label{ss:2pf} \  We write $\{ \mscr{F}_p \}$ as a
two-parameter family $\{ \mscr{F}_{a, b_{n}} \}$ of 2D maps. Both
$a$ and $b_{n}$ are derived from $\mu = e^p$ as follows. Let
$\mu_0
>0 $ be sufficiently small. Define $\gamma: (0, \mu_0] \to {\mathbb R}$
via $\gamma(\mu) = \frac{\omega}{\beta} \ln \mu^{-1}$. For $n \in
{\mathbb Z}^+$ satisfying $n \geqs (2 \pi \beta)^{-1} \omega \ln
\mu_0^{-1}$, let $\mu_n \in (0, \mu_0]$ be such that
$\gamma(\mu_n) = n$. Notice that $\mu_n \to 0$ monotonically. Set
$b_n = \mu_n$. For $\mu \in (\mu_{n+1}, \mu_n]$ and $a \in [0, 2
\pi) = \mbb{S}^{1}$, we define
\begin{equation*}
\mu(n, a) = \gamma^{-1}(\gamma(\mu_n) + a) = \mu_n
e^{-\tfrac{\beta}{\omega} a}
\end{equation*}
and
\begin{equation*}
p(n, a) = \ln \mu(n, a) = \ln \mu_n - \frac{\beta}{\omega} a.
\end{equation*}
Define
\begin{equation*}
\mscr{F}_{a, b_n} = \mscr{F}_{p(n, a)}.
\end{equation*}

\smallskip

\noindent {\bf Verification of~\pref{c1}--\pref{c4}.}  We
prove Theorem \ref{th-s1.1} by applying Propositions~\ref{prop3-s3.2}
and~\ref{prop4-s3.2}. We verify~\bpref{c1}--\bpref{c4} for $\mscr{F}_{a, b_n}$.
Proposition~\ref{prop1-s4.2} establishes~\bpref{c1}.

\begin{proposition}\label{prop1-s4.2}
We have
\begin{equation}
\|\mscr{F}_{a, b_n}({\mathbb X}, \theta) - (0, \mscr{F}_{a, 0} ({\mathbb
X}, \theta))\|_{C^3(\Sigma^- \times [0, 2\pi))} \to 0
\end{equation}
as $b_n \to 0$, where
\begin{equation}\label{f1-F}
\begin{split}
\mscr{F}_{a,0} (\mbb{X}, \thet) = \thet &+ \om (L^+ + L^-) + a + \frac{\om}{\be} \ln (\ve K_{1} (\ve)^{-1}) \\
&{}- \frac{\om}{\be} \ln \bigg[ \left( 1 + \mcal{O}_{\thet, p} (\ve) \right) \bigg( 1 + c_{1} \sin (\thet) + c_{2} \cos (\thet) \\
&\hspace{1.7cm} {}+ \frac{P_{L}}{K_{1} (\ve)} \left( \mbb{X} +
\mcal{O}_{\thet, p} (1) + \mcal{O}_{\mbb{X}, \thet, p} (\ve) \right) \bigg)
- K_{1} (\ve)^{-1} \mcal{O}_{\thet, p} (1) \bigg].
\end{split}
\end{equation}
\end{proposition}

\begin{proof}
The only problematic term in~\pref{e:fthet} has the form
\begin{equation*}
\frac{\omega}{\beta + {\mathcal O}_{{\mathbb X}_0, \theta_0,
p}(\mu)} \ln \mu^{-1},
\end{equation*}
which we write as
\begin{equation*}
\frac{\omega}{\beta} \ln \mu^{-1} + \frac{\omega \cdot {\mathcal
O}_{{\mathbb X}_0, \theta_0, p}(\mu)}{\beta(\beta + {\mathcal
O}_{{\mathbb X}_0, \theta_0, p}(\mu))} \ln \mu^{-1}.
\end{equation*}
Observe that the $C^3$ norm of the second term $\to 0$ as $b_n \to
0$ and the first term may be computed modulo $2 \pi$ and is
therefore equal to $a$.  Viewing $\mu$ as a function of $a$, the
$C^3$ norm of ${\mathbb X}_1$ is bounded by
\begin{equation*}
K( \varepsilon) \mu^{\tfrac{\tilde \alpha}{\tilde \beta}-1}
\end{equation*}
and therefore decays to $0$ as $b_n \to 0$ provided that~\bpref{li:h1}\bpref{li:h1b} holds.
\end{proof}

For~\bpref{c2} we apply Proposition~\ref{prop2-s4.2} to the
family of circle maps
\begin{equation}\label{f1a-F}
\begin{split}
\mscr{F}_{a,0} (0, \thet) = \thet &+ \om (L^+ + L^-) + a + \frac{\om}{\be} \ln (\ve K_{1} (\ve)^{-1}) \\
&{}- \frac{\om}{\be} \ln \bigg[ \left( 1 + \mcal{O}_{\thet, p}
(\ve) \right) \bigg( 1 + c_{1} \sin (\thet) + c_{2} \cos (\thet) \\
&\hspace{1.7cm} {}+ \frac{P_{L}}{K_{1} (\ve)} (\mcal{O}_{\thet, p}
(1) + \mcal{O}_{\mbb{X}, \thet, p} (\ve)) \bigg) - K_{1} (\ve)^{-1} \mcal{O}_{\thet, p}(1) \bigg].
\end{split}
\end{equation}
To apply Proposition~\ref{prop2-s4.2} to the family $\{
\mscr{F}_{a, 0} (0, \thet) \}$, we set
\begin{gather*}
{\mathcal K} = \frac{\om}{\be}\\
\Psi(\theta) = - \ln (1 + c_1 \sin \theta + c_2 \cos \theta)\\
\Phi (\thet, a) = \mscr{F}_{a, 0} (0, \thet) - \ga - \thet - a - \mscr{K}
\Psi (\thet)
\end{gather*}
where
\begin{equation*}
\ga = \om (L^{+} + L^{-}) + \frac{\om}{\be} \ln (\ve K_{1} (\ve)^{-1}).
\end{equation*}
The assumption on the $C^{3}$ norm of $\Phi$ is satisfied if $\ve$
is sufficiently small.

Hypothesis~\bpref{c3} follows directly from ~\pref{f1-F}.  Hypothesis~\bpref{c4} follows from a
direct computation using (\ref{e:fprelim}).  Finally, to apply Proposition~\ref{prop4-s3.2}
we need to verify that $\lambda_0>\ln 10$.
This follows if $\omega$ is sufficiently large.  The
proof of Theorem \ref{th-s1.1} is complete.

\section{Computational proofs}\label{s6}

\subsection{Proof of Proposition \ref{prop-2.4}}\label{sadd-1} \
Let ${\mathbb F} = {\mathbb F}({\mathbb X}, {\mathbb Y}, \theta;
\mu)$ and ${\mathbb G} = {\mathbb G}({\mathbb X}, {\mathbb Y},
\theta; \mu)$ be as in equation ~(\ref{f1-s2.1d}). For a
combination ${\mathbb Z} = {\mathbb X}^{d_1} {\mathbb Y}^{d_2}
\mu^{d_3}$ of powers of the variables ${\mathbb X}$, ${\mathbb
Y}$, and $\mu$, let $\partial^k_{\mathbb Z}$ denote the
corresponding partial derivative operator, where $k = d_1 + d_2 +
d_3$ is the order. There exists $K_3 > 0$ such that for every
${\mathbb Z}$ of order $\leqs 3$ and $0 \leqs i \leqs 3$, we have
\begin{equation}\label{sharp-de}
|\partial_{\mathbb Z}^k (\partial_{\theta^i}^i {\mathbb F} \cdot
{\mathbb Z})| < K_3, \quad |\partial_{\mathbb Z}^k
(\partial_{\theta^i}^i {\mathbb G} \cdot {\mathbb Z})| < K_3
\end{equation}
on ${\mathbb D}_{t, q_0, p}$. This is because the $C^3$ norms of
${\bf F}({\bf X}, {\bf Y}, \theta; \mu)$ and ${\bf G}({\bf X},
{\bf Y}, \theta; \mu)$ are bounded on $\mscr{U}_{\varepsilon}
\times [0, \mu_0]$ and because ${\mathbb F}({\mathbb X}, {\mathbb
Y}, \theta; \mu) = {\bf F}(\mu {\mathbb X}, \mu {\mathbb Y},
\theta; \mu)$ and ${\mathbb G}({\mathbb X}, {\mathbb Y}, \theta;
\mu) = {\bf G}(\mu {\mathbb X}, \mu {\mathbb Y}, \theta; \mu)$.

\vspace{0.1cm}

\noindent {\mathversion{bold}  $C^0$} {\bfseries estimates.} Using
~(\ref{sharp-de}) with $i = k = 0$ and ~(\ref{f5-s2.2}), we have
\begin{equation}\label{f3-s2.3}
\|U \|_{C^0({\mathbb D}_{t, q_0, p})} < K_{3} \mu, \quad \|V
\|_{C^0({\mathbb D}_{t, q_{0}, p})} < K_3 \mu.
\end{equation}

\vspace{0.1cm}

\noindent {\mathversion{bold} $C^1$} {\bfseries estimates.} We now
estimate the first derivatives.

\vspace{0.1cm}

\noindent {\sffamily On $\partial_{{\mathbb Y}_0} U$ and
$\partial_{{\mathbb Y}_0} V$.} \ Using $\theta(t) = \theta_0 +
\omega t$, we have $\partial_{{\mathbb Y}_0} \theta = 0$. Using
~(\ref{f5-s2.2}), we have
\begin{align}
\sublabon{equation}
\partial_{{\mathbb Y}_0} U &= \mu t^{-1} \int_0^t
\left( \partial_{{\mathbb X}} {\mathbb F} \cdot \partial_{{\mathbb
Y}_0} {\mathbb X} + \partial_{\mathbb Y}{\mathbb F} \cdot
\partial_{{\mathbb Y}_0}
{\mathbb Y} \right) ds
\label{spatialU} \\
\partial_{{\mathbb Y}_0} V &= \mu t^{-1} \int_0^t
\left( \partial_{{\mathbb X}} {\mathbb G} \cdot \partial_{{\mathbb
Y}_0} {\mathbb X} + \partial_{\mathbb Y} {\mathbb G} \cdot
\partial_{{\mathbb Y}_0} {\mathbb Y} \right) ds.
\label{spatialV}
\end{align}
\sublaboff{equation} To make these formulas useful, we need to
write $\partial_{{\mathbb Y}_0} {\mathbb X}$ and
$\partial_{{\mathbb Y}_0} {\mathbb Y}$ in terms of
$\partial_{{\mathbb Y}_0} U$ and $\partial_{{\mathbb Y}_0} V$. For
this purpose we use ~(\ref{f4-s2.2}). We have
\begin{equation}\label{f5-s2.3}
\begin{aligned}
\partial_{{\mathbb Y}_0} {\mathbb X} &= t {\mathbb X}
\partial_{{\mathbb Y}_0} U \\
\partial_{{\mathbb Y}_0} {\mathbb Y} &= t {\mathbb Y}
\partial_{{\mathbb Y}_0} V + \frac{\mathbb Y}{{\mathbb Y}_0}.
\end{aligned}
\end{equation}
Combining~(\ref{spatialU}),~\pref{spatialV}, and~(\ref{f5-s2.3}),
we obtain
\begin{equation}\label{f6-s2.3}
\begin{aligned}
\partial_{{\mathbb Y}_0} U &= \mu t^{-1} \int_0^t
\left( \partial_{{\mathbb X}} {\mathbb F} \cdot {\mathbb X} \cdot
s \partial_{{\mathbb Y}_0} U + \partial_{\mathbb Y}{\mathbb F}
\cdot {\mathbb Y} \cdot s
\partial_{{\mathbb Y}_0} V \right) ds + \mu t^{-1} \int_0^t
\partial_{\mathbb Y}{\mathbb F} \cdot \frac{\mathbb Y}{{\mathbb Y}_0} \, ds
\\
\partial_{{\mathbb Y}_0} V &= \mu t^{-1} \int_0^t
\left( \partial_{{\mathbb X}} {\mathbb G} \cdot {\mathbb X} \cdot
s \partial_{{\mathbb Y}_0} U + \partial_{\mathbb Y}{\mathbb G}
\cdot {\mathbb Y} \cdot s \partial_{{\mathbb Y}_0} V \right) ds +
\mu t^{-1} \int_0^t
\partial_{\mathbb Y}{\mathbb G}
\cdot \frac{\mathbb Y}{{\mathbb Y}_0} \, ds.
\end{aligned}
\end{equation}
Using ~(\ref{sharp-de}), we have
\begin{equation*}
|\partial_{\mathbb X} {\mathbb F} \cdot {\mathbb X}| < K_{3},
\quad |\partial_{\mathbb X} {\mathbb G} \cdot {\mathbb X}| <
K_{3}, \quad |\partial_{\mathbb Y} {\mathbb F} \cdot {\mathbb Y}|
< K_{3}, \quad |\partial_{\mathbb Y}{\mathbb G} \cdot {\mathbb Y}|
< K_3.
\end{equation*}
Using ~(\ref{f6-s2.3}), we have
\begin{equation}\label{f7-s2.3}
\begin{aligned}
|\partial_{{\mathbb Y}_0} U| &\leqs K \mu t^{-1} \int_0^t \left(|s
\partial_{{\mathbb Y}_0} U| + |s
\partial_{{\mathbb Y}_0} V| \right) ds + K \mu \\
|\partial_{{\mathbb Y}_0} V| &\leqs K \mu t^{-1} \int_0^t \left(
|s
\partial_{{\mathbb Y}_0} U| + |s \partial_{{\mathbb Y}_0} V|
\right) ds + K \mu,
\end{aligned}
\end{equation}
from which it follows that
\begin{equation*}
|\partial_{{\mathbb Y}_0} U| < K \mu, \quad |\partial_{{\mathbb
Y}_0} V| < K \mu.
\end{equation*}

\vspace{0.1cm}

\noindent {\sffamily On $\partial_{{\mathbb X}_0} U$ and
$\partial_{{\mathbb X}_0} V$.} Mimic the proof above.

\vspace{0.1cm}

\noindent {\sffamily On $\partial_{\theta_0} U$ and
$\partial_{\theta_0} V$.} We follow similar lines of computation.
Since $\partial_{\theta_0} \theta = 1$, we have
\begin{align*}
\partial_{\theta_0} U &= \mu t^{-1} \int_0^t
\left( \partial_{{\mathbb X}} {\mathbb F} \cdot
\partial_{\theta_0} {\mathbb X} + \partial_{\mathbb Y}{\mathbb
F} \cdot
\partial_{\theta_0}
{\mathbb Y}  + \partial_{\theta} {\mathbb F}  \right) ds \\
\partial_{\theta_0} V &= \mu t^{-1} \int_0^t
\left( \partial_{{\mathbb X}} {\mathbb G} \cdot
\partial_{\theta_0} {\mathbb X} + \partial_{\mathbb Y} {\mathbb
G} \cdot
\partial_{\theta_0} {\mathbb Y}
+ \partial_{\theta} {\mathbb G} \right) ds.
\end{align*}
Analogous to ~(\ref{f5-s2.3}), we have
\begin{equation*}
\partial_{\theta_0} {\mathbb X} = t {\mathbb X}
\partial_{\theta_0} U, \quad \partial_{\theta_0} {\mathbb Y} = t {\mathbb Y}
\partial_{\theta_0} V.
\end{equation*}
Arguing as above, we conclude that
\begin{equation*}
|\partial_{\theta_0} U| < K \mu, \quad |\partial_{\theta_0} V| < K
\mu.
\end{equation*}

\vspace{0.1cm}

\noindent {\sffamily On $\partial_p U$ and $\partial_p V$.} We
follow similar lines of computation. Note that we have
\begin{equation*}
\partial_p \mu = \mu, \quad \partial_p {\mathbb F} = \mu
\partial_{\mu} {\mathbb F},
\end{equation*}
and so on. Starting with ~(\ref{f5-s2.2}), we have
\begin{equation}\label{f8-s2.3}
\begin{aligned}
\partial_{p} U &= \mu t^{-1} \int_0^t {\mathbb F} \, ds +
\mu t^{-1} \int_0^t \left( \partial_{{\mathbb X}} {\mathbb F}
\cdot \partial_{p} {\mathbb X} + \partial_{\mathbb Y}{\mathbb F}
\cdot \partial_{p}
{\mathbb Y} + \mu \partial_{\mu} {\mathbb F} \right) ds \\
\partial_{p} V &= \mu t^{-1} \int_0^t {\mathbb G} \, ds +
\mu t^{-1} \int_0^t \left( \partial_{{\mathbb X}} {\mathbb G}
\cdot
\partial_{p} {\mathbb X} + \partial_{\mathbb Y}{\mathbb G} \cdot
\partial_{p} {\mathbb Y} + \mu \partial_{\mu} {\mathbb G} \right)
ds
\end{aligned}
\end{equation}
and using~(\ref{f4-s2.2}) we have
\begin{equation}\label{f9-s2.3}
\begin{aligned}
\partial_p {\mathbb X} &= t {\mathbb X} \partial_p U  \\
\partial_p {\mathbb Y} &= t {\mathbb Y} \partial_p V.
\end{aligned}
\end{equation}
Now argue as above.

\vspace{0.1cm}

\noindent {\sffamily On $\partial_t U$ and $\partial_t V$.} \ The
partial derivatives of $U$ and $V$ with respect to $t$ are easier
to estimate because when differentiating with respect to $t$ using
~(\ref{f5-s2.2}), no derivatives are involved on the right-hand
side so the estimates on $\pdop{t} U$ and $\pdop{t} V$ are
obtained directly from $C^0$ estimates. We have
\begin{equation*}
|\partial_t U| < K \mu, \quad |\partial_t V| < K \mu.
\end{equation*}
This completes the desired estimates on the first derivatives.

\vspace{0.1cm}

\noindent {\mathversion{bold} $C^2$} {\bfseries estimates.} We now
move to the second derivatives. We estimate $\partial^2_{{\mathbb
Y}_0 {\mathbb Y}_0}U$ and $\partial^2_{{\mathbb Y}_0 {\mathbb
Y}_0} V$ first.  Using~(\ref{spatialU}), we have
\begin{align*}
\partial_{{\mathbb Y}_0 {\mathbb Y}_0}^2 U &= \mu t^{-1} \int_0^t
\left( \partial_{{\mathbb X}{\mathbb X}}^2 {\mathbb F} \cdot
(\partial_{{\mathbb Y}_0} {\mathbb X})^2 + 2 \partial_{{\mathbb
X}{\mathbb Y}}^2 {\mathbb F} \cdot (\partial_{{\mathbb Y}_0}
{\mathbb X})(\partial_{{\mathbb Y}_0} {\mathbb Y}) +
\partial_{{\mathbb Y}{\mathbb Y}}({\mathbb F} \cdot
\partial_{{\mathbb Y}_0}
{\mathbb Y})^2  \right) ds \\
&\qquad {}+ \mu t^{-1} \int_0^t \left( \partial_{{\mathbb X}} {\mathbb F}
\cdot \partial_{{\mathbb Y}_0 {\mathbb Y}_0}^2 {\mathbb X} +
\partial_{\mathbb Y}{\mathbb F} \cdot
\partial_{{\mathbb Y}_0{\mathbb Y}_0}^2
{\mathbb Y} \right) ds.
\end{align*}
Using ~(\ref{f5-s2.3}), we have
\begin{equation}\label{f1-c2}
\begin{aligned}
\partial_{{\mathbb Y}_0 {\mathbb Y}_0}^2 {\mathbb X} &=
t \partial_{{\mathbb Y}_0} {\mathbb X} \cdot \partial_{{\mathbb
Y}_0} U + t {\mathbb X}
\partial_{{\mathbb Y}_0 {\mathbb Y}_0}^2 U \\
\partial_{{\mathbb Y}_0 {\mathbb Y}_0}^2 {\mathbb Y} &=
t \partial_{{\mathbb Y}_0} {\mathbb Y} \cdot
\partial_{{\mathbb Y}_0} V  + t {\mathbb Y} \cdot
\partial_{{\mathbb Y}_0 {\mathbb Y}_0} V
+ \frac{\partial_{{\mathbb Y}_0} \mathbb Y}{{\mathbb Y}_0} -
\frac{\mathbb Y}{{\mathbb Y}_0^2}.
\end{aligned}
\end{equation}
Therefore, $\partial_{{\mathbb Y}_0 {\mathbb Y}_0}^2 U$ is given
by
\begin{equation}\label{f2-c2}
\begin{split}
\partial_{{\mathbb Y}_0 {\mathbb Y}_0}^2 U &= \mu t^{-1} \int_0^t
\left( \partial_{{\mathbb X}{\mathbb X}}^2 {\mathbb F} \cdot
(\partial_{{\mathbb Y}_0} {\mathbb X})^2 + 2 \partial_{{\mathbb
X}{\mathbb Y}}^2 {\mathbb F} \cdot (\partial_{{\mathbb Y}_0}
{\mathbb X})(\partial_{{\mathbb Y}_0} {\mathbb Y}) +
\partial_{{\mathbb Y}{\mathbb Y}}({\mathbb F} \cdot
\partial_{{\mathbb Y}_0}
{\mathbb Y})^2  \right) ds \\
&\qquad {}+ \mu t^{-1} \int_0^t \left( \partial_{{\mathbb X}} {\mathbb F}
\cdot \partial_{{\mathbb Y}_0} {\mathbb X} \cdot s
\partial_{{\mathbb Y}_0} U +
\partial_{\mathbb Y}{\mathbb F} \cdot
\partial_{{\mathbb Y}_0} {\mathbb Y} \cdot
s \partial_{{\mathbb Y}_0} V \right) ds \\
&\qquad {}+ \mu t^{-1}\int_0^t \partial_{\mathbb Y}{\mathbb F}\cdot
\left(\frac{\partial_{{\mathbb Y}_0} \mathbb Y}{{\mathbb Y}_0} -
\frac{\mathbb Y}{{\mathbb Y}_0^2} \right) ds\\
&\qquad {}+ \mu t^{-1} \int_0^t \left( \partial_{{\mathbb X}} {\mathbb F}
\cdot {\mathbb X} \cdot s \partial_{{\mathbb Y}_0 {\mathbb Y}_0}^2
U + \partial_{\mathbb Y}{\mathbb F} \cdot {\mathbb Y} \cdot s
\partial_{{\mathbb Y}_0{\mathbb Y}_0}^2 V \right) ds.
\end{split}
\end{equation}
To estimate the first three integrals in~\pref{f2-c2}, we use
~(\ref{f5-s2.3}) for $\partial_{{\mathbb Y}_0} {\mathbb X}$ and
$\partial_{{\mathbb Y}_0} {\mathbb Y}$. Using the first derivative
estimates and using ~(\ref{sharp-de}) repeatedly, we bound these
integrals by $K \mu$. Note that we also need ${\mathbb Y}_0 > 1$
(see Corollary ~\ref{coro1-s2.3c}) for the third integral.
Together with an analogous formula for $\partial_{{\mathbb Y}_0
{\mathbb Y}_0}^2 V$ in which we replace ${\mathbb F}$ with
${\mathbb G}$, we conclude that
\begin{equation*}
|\partial_{{\mathbb Y}_0 {\mathbb Y}_0}^2 U| < K \mu, \quad
|\partial_{{\mathbb Y}_0 {\mathbb Y}_0}^2 V| < K \mu.
\end{equation*}

All other second derivatives are estimated similarly. Here we skip
the details to avoid repetitive computations.

\vspace{0.1cm}

\noindent {\mathversion{bold} $C^3$} {\bfseries estimates.} Third
derivatives are estimated in the same spirit. Since the formulas
for a given third derivative depend on previous computations of
relevant second derivatives, here we estimate $\partial_{{\mathbb
Y}_0 {\mathbb Y}_0 p}^3 U$ and $\partial_{{\mathbb Y}_0 {\mathbb
Y}_0 p}^3 V$ as a representative example. Of all of the third
derivatives, these are the most tedious to compute.

To compute $\partial_{{\mathbb Y}_0 {\mathbb Y}_0 p}^3 U$ we apply
$\partial_p$ to ~(\ref{f2-c2}). The explicit factor $\mu$ written
in front of all integrals generates a collection of terms that is
identical to the right-hand side of ~(\ref{f2-c2}).  We showed
when estimating second derivatives that the size of each of these
terms in bounded by $K \mu$.

The remaining terms are produced by applying $\partial_p$ to the
functions inside of the integrals in ~(\ref{f2-c2}). The terms
produced from the first three integrals are estimated using the
$C^2$ estimates. Estimate ~(\ref{sharp-de}) is used repeatedly. It
is critically important that potentially problematic terms in the
form of powers of ${\mathbb Y}$ and ${\mathbb X}$, introduced by
using the likes of ~(\ref{f5-s2.3}),~(\ref{f9-s2.3}), and
~(\ref{f1-c2}), are always matched perfectly with corresponding
partial derivatives with respect to ${\mathbb F}$ or ${\mathbb
G}$. Applying $\partial_p$ to the fourth integral, we obtain an
integral term of the form
\begin{equation*}
(I) = \mu t^{-1} \int_0^t \left( \partial_{{\mathbb X}} {\mathbb
F} \cdot {\mathbb X} \cdot s \partial_{{\mathbb Y}_0 {\mathbb Y}_0
p}^3 U + \partial_{{\mathbb Y}}{\mathbb F} \cdot {\mathbb Y} \cdot
s
\partial_{{\mathbb Y}_0{\mathbb Y}_0 p}^3 V \right) ds
\end{equation*}
and a collection of other terms that can be treated the same way
as the terms produced by differentiating the first three
integrals. We have
\begin{equation*}
|(I)| \leqs K \mu t^{-1} \int_0^t \left( |s \partial_{{\mathbb
Y}_0 {\mathbb Y}_0 p}^3 U| + |s
\partial_{{\mathbb Y}_0{\mathbb Y}_0 p}^3 V| \right) ds.
\end{equation*}
Combining this analysis with analogous estimates for
$|\partial_{{\mathbb Y}_0{\mathbb Y}_0 p}^3 V|$, we obtain
\begin{equation*}
|\partial_{{\mathbb Y}_0 {\mathbb Y}_0 p}^3 U| < K \mu, \quad
|\partial_{{\mathbb Y}_0{\mathbb Y}_0 p}^3 V| < K \mu.
\end{equation*}
This completes the proof of Proposition ~\ref{prop-2.4}.

\subsection{Proof of Proposition \ref{prop1-s3.3}}\label{sadd-2}
The proof of this proposition is lengthy because of the
complicated composition process explained earlier in Sect.
\ref{s3.3}.

\vspace{0.1cm}

\noindent {\mathversion{bold} $C^0$} {\bfseries estimates.}  We
first establish a $C^{0}$ control on $T$.

\begin{lemma}\label{lem1-T}
There exist constants $K_4 < K_5$ independent of $\varepsilon$
such that for all $q_0 \in \Sigma^+$, we have $K_4 \ln \mu^{-1} <
T(q_0; \mu) < K_5 \ln \mu^{-1}$.
\end{lemma}

\begin{proof}[Proof of Lemma~\ref{lem1-T}]
Using
\begin{equation*}
{\mathbb Y}(T) = {\mathbb Y}_0 e^{(\beta + V(T))T}
\end{equation*}
we obtain
\begin{equation*}
T = \frac{1}{\beta + V(T)} \ln \frac{{\mathbb Y}(T)}{{\mathbb
Y}_0}.
\end{equation*}
Since $({\mathbb X}(T), {\mathbb Y}(T), \theta(T))$ is on
$\Sigma^-$, Proposition ~\ref{prop-s2.3b} implies that
\begin{equation*}
{\mathbb Y}(T) \approx  \mu^{-1} \varepsilon
\end{equation*}
and the desired estimates follow from $|V(T)|< K \mu$ and $1 <
{\mathbb Y}_0 < K(\varepsilon)$.
\end{proof}

\begin{lemma}\label{lem2-T}
We have $\mu^{-1} e^{-\alpha T} < 1$.
\end{lemma}

\begin{proof}[Proof of Lemma~\ref{lem2-T}]
We substitute
\begin{equation*}
T = \frac{1}{\beta + V(T)} \ln \frac{{\mathbb Y}(T)}{{\mathbb
Y}_0}
\end{equation*}
into~(\ref{f1-T}) to obtain
\begin{equation*}
{\mathbb X}(T) = \left( \frac{{\mathbb Y}_0}{{\mathbb
Y}(T)}\right)^{\frac{\alpha - U(T)}{\beta+V(T)}} {\mathbb X}_0.
\end{equation*}
We then use ${\mathbb Y}(T) \approx \varepsilon \mu^{-1}$,
${\mathbb X}_0 \approx \varepsilon \mu^{-1}$, $|U(T)| < K \mu$,
$|V(T)| < K \mu$, and $\alpha > \beta$ to conclude that ${\mathbb
X}(T) \ll \varepsilon$. We have
\begin{equation*}
\frac{1}{10} \varepsilon \mu^{-1} e^{-\alpha T} < {\mathbb X}_0
e^{(- \alpha + U(T))T} = {\mathbb X}(T) \ll \varepsilon.
\end{equation*}
For the first inequality, we use ${\mathbb X}_0 \approx
\varepsilon \mu^{-1}$ and $|U(T) T| < K \mu \ln \mu^{-1} \ll 1$.
This proves the lemma.
\end{proof}

\vspace{0.1cm}

\noindent {\mathversion{bold} $C^1$} {\bfseries estimates.}  We
present $C^{1}$ estimates with respect to $(Z_0, \theta_0, p)$,
where $(Z_0, \theta_0) \in \Sigma^+$ and $p \in (-\infty, \ln
\mu_0]$.

\begin{lemma}\label{lem3-T}
There exist constants $K_7$ and $K_8$ independent of $\varepsilon$
such that
\begin{equation*}
\|{\mathbb X}(T) \|_{C^1} < K_7 + K_8 \|T\|_{C^1}, \quad
\|\theta(T)\|_{C^1} < K_7 + K_8 \|T\|_{C^1}.
\end{equation*}
\end{lemma}

\begin{proof}[Proof of Lemma~\ref{lem3-T}]
The bound on $\theta(T)$ is trivial because $\theta(T) = \theta_0
+ \omega T$. For ${\mathbb X}(T)$, we have
\begin{align*}
{\mathbb X}(T) &= {\mathbb X}_0 e^{(-\alpha + U(T))T} \\
&= \varepsilon \mu^{-1} e^{(-\alpha + U(T, {\mathbb X}_0, {\mathbb
Y}_0, \theta_0; p))T}+ {\mathcal O}_{Z_0, \theta_0,
p}(1)e^{(-\alpha + U(T, {\mathbb X}_0, {\mathbb Y}_0, \theta_0;
p))T}.
\end{align*}
Notice that for the second equality, ~(\ref{f1-s3.3}) is used for
${\mathbb X}_0$. We regard ${\mathbb X}_0$ and ${\mathbb Y}_0$ as
functions of $Z_0$, $\theta_0$, and $p$ defined by
~(\ref{f1-s3.3}). The desired estimate follows from using
Proposition ~\ref{prop-2.4} for $U$ and ~(\ref{f1-s3.3}) for
${\mathbb X}_0$ and ${\mathbb Y}_0$. We also use Lemma
~\ref{lem2-T}.
\end{proof}

\begin{lemma}\label{lem4-T}
We have
\begin{equation*}
\|T - \frac{1}{\beta} \ln \mu^{-1}\|_{C^1} < K.
\end{equation*}
\end{lemma}

\begin{proof}[Proof of Lemma~\ref{lem4-T}]
Using ~(\ref{f2-T}), we write ~(\ref{f3-T}) as
\begin{equation*}
\mu^{-1} (\varepsilon + \mu {\bf f}({\mathbb X}(T), \theta(T); p))
= {\mathbb Y}_0 e^{(\beta + V(T))T}.
\end{equation*}
Solving for $T$, we obtain
\begin{equation}\label{f4-T}
\begin{split}
T - \frac{1}{\beta} \ln \mu^{-1} &= - \frac{V(T)}{\beta(\beta +
V(T))}\ln \mu^{-1} - \frac{1}{\beta + V(T)} \ln {\mathbb Y}_0 \\
&\qquad + \frac{1}{\beta + V(T)} \ln (\varepsilon + \mu {\bf
f}({\mathbb X}(T), \theta(T); p)).
\end{split}
\end{equation}
In~\pref{f4-T}, $V(T) = V(T, {\mathbb X}_0, {\mathbb Y}_0,
\theta_0; p)$, and ${\mathbb X}_0$ and ${\mathbb Y}_0$ are written
in terms of $Z_0$, $\theta_0$, and $p$ using ~(\ref{f1-s3.3}).
Using Proposition ~\ref{prop-2.4}, we have
\begin{equation*}
\|T - \frac{1}{\beta} \ln \mu^{-1}\|_{C^0} < K.
\end{equation*}

First derivatives of $T$ are estimated by directly differentiating
~(\ref{f4-T}). We estimate $\partial_{Z_0} T$ as a representative
example. Differentiating ~(\ref{f4-T}), we have
\begin{equation*}
\partial_{Z_0} T = (I) + (II) \partial_{Z_0} T,
\end{equation*}
where (I) is a collection of terms that do not depend on
$\partial_{Z_0} T$ and (II) is a function of $Z_0$, $\theta_0$,
and $p$. Using Proposition ~\ref{prop-2.4} for $V(T)$,
~(\ref{f1-s3.3}) for ${\mathbb X}_0$ and ${\mathbb Y}_0$, and
Lemma ~\ref{lem3-T} for $\partial_{Z_0} {\mathbb X} (T)$ and
$\partial_{Z_0} \theta(T)$, we have $|(I)| < K$ and $|(II)| \ll
1$.
\end{proof}

\vspace{0.1cm}

\noindent {\bfseries Higher derivative estimates.} With the first
derivatives controlled by Lemmas ~\ref{lem3-T} and ~\ref{lem4-T},
we estimate the second derivatives by first proving a version of
Lemma ~\ref{lem3-T} and then proving a version of Lemma
~\ref{lem4-T} for the $C^2$ norms. We then do the same for the
$C^3$ norms.  This completes the proof of
Proposition~\ref{prop1-s3.3}.

\bibliographystyle{amsplain}
\bibliography{dhl_lit}

\end{document}